%% file: revised.tex
\documentclass[12pt]{amsart}
\usepackage{amsmath, amssymb, latexsym, amsthm,bm}
\usepackage{mathtools}
\usepackage{url}
\usepackage[T1]{fontenc}
\usepackage{mathrsfs}
\usepackage{color}
\usepackage{tikz}
\usepackage{graphicx}
\usepackage{stmaryrd,amsbsy,enumerate}

\usepackage{mathscinet}
\usetikzlibrary{calc}
\usepackage{cancel}
\usepackage{xspace}
\usepackage{url}
\usepackage{longtable}
\usepackage[obeyFinal,colorinlistoftodos,textsize=tiny]{todonotes}

\usepackage{tikz}

\usepackage[font=small]{caption}
\usepackage{mathtools,amsmath}
\usepackage{array}

\usepackage{tikz}

\usepackage{stmaryrd}

\def\tilp{{\widetilde{p}}}

\def\ucr#1{{\text{\rm cr}(#1)}}
\def\pcr#1{{\text{\rm cr}_P(#1)}}

\DeclarePairedDelimiterX{\infdivx}[2]{(}{)}{%
  #1\;\delimsize\|\;#2%
}

\def\Va{{v_{i+(\beta k + 1)}}}
\def\Vb{{v_{i+2k}}}
\def\Vc{{v_{i+\beta k}}}
\def\Vd{{u_{i+\alpha k}}}
\def\Ve{{u_{i+(\alpha k+1)}}}
\def\Vf{{u_{i+2k}}}
\def\Vg{{u_{i+1}}}
\def\Vh{{v_{i+1}}}
\def\Vi{{u_{i{-}2k}{=}u_{i{+}(2k{+}1)}}}
\def\Vj{{v_{i-2k}}}
\def\Vk{{{=}v_{i{+}(2k{+}1)}}}
\def\Vl{{u_{i-1}}}
\def\Vm{{u_{i-(\alpha k-1)}}}
\def\Vn{{v_{i-1}}}
\def\Vo{{u_{i-\alpha k}}}
\def\Vp{{v_{i-\beta k}}}
\def\Vq{{v_{i-(\beta k -1)}}}

\renewcommand{\descriptionlabel}[1]%
  {{\hglue -0.7 cm}\hspace{\labelsep}#1}

\def\ignore#1{{ }}

\def\floor#1{\lfloor{#1}\rfloor}

\definecolor{mygray}{gray}{0.8}

\def\real{{\mathbb R}}

\def\cc{{\mathscr C}}

\def\myfrac#1#2{{\genfrac{}{}{0pt}{}{#1}{#2}}}

\def\cir#1#2#3{{\cc{\hbox{\hglue #1 cm}\myfrac{#2}{#3}}}}

\def\centerarc[#1](#2)(#3:#4:#5){ \draw[#1] ($(#2)+({#5*cos(#3)},{#5*sin(#3)})$) arc (#3:#4:#5); } 

\def\conethree{\cir{-0.04}{1}{3}}

\def\cthreefive{\cir{-0.02}{3}{5}}

\def\c13{{\conethree}}
\def\c35{{\cthreefive}}

\newtheorem{theorem}{Theorem} 
\newtheorem{theorem*}{Theorem} 

\newtheorem{observation}[theorem]{Observation}

\newtheorem{claim}[theorem]{Claim}
\newtheorem{remark}[theorem]{Remark}

\theoremstyle{definition}

\begin{document}



\title[Drawings of complete graphs in the projective plane]{Drawings of complete graphs in the projective plane}

\author[Arroyo]{Alan Arroyo}
\address{IST Austria, 3400 Klosterneuburg, Austria}
\email{\tt alanmarcelo.arroyoguevara@ist.ac.at}

\author[McQuillan]{Dan McQuillan}
\address{College of Science and Mathematics, Norwich University. Northfield, VT 05663}
\email{\tt dmcquill@norwich.edu}

\author[Richter]{R.~Bruce Richter}
\address{Department of Combinatorics and Optimization. University of Waterloo. Waterloo, ON Canada N2L 3G1}
\email{\tt brichter@uwaterloo.ca}

\author[Salazar]{Gelasio Salazar}
\address{Instituto de F\'\i sica, Universidad Aut\'onoma de San Luis Potos\'{\i}, SLP 78000, Mexico}
\email{\tt gsalazar@ifisica.uaslp.mx}

\author[Sullivan]{Matthew Sullivan}
\address{Department of Combinatorics and Optimization. University of Waterloo. Waterloo, ON Canada N2L 3G1}
\email{\tt m8sullivan@uwaterloo.ca}




\begin{abstract}
Hill's Conjecture states that the crossing number $\ucr{K_n}$ of the complete graph $K_n$ in the plane (equivalently, the sphere) is $\frac{1}{4}\floor{\frac{n}{2}}\floor{\frac{n-1}{2}}\floor{\frac{n-2}{2}}\floor{\frac{n-3}{2}}=n^4/64 + O(n^3)$. Moon proved that the expected number of crossings in a spherical drawing in which the points are randomly distributed and joined by geodesics is precisely $n^4/64+O(n^3)$, thus matching asymptotically the conjectured value of $\ucr{K_n}$. Let $\pcr{G}$ denote the crossing number of a graph $G$ in the projective plane. Recently, Elkies proved that the expected number of crossings in a naturally defined random projective plane drawing of $K_n$ is $(n^4/8\pi^2)+O(n^3)$. In analogy with the relation of Moon's result to Hill's conjecture, Elkies asked if $\lim_{n\to\infty} \pcr{K_n}/n^4=1/8\pi^2$. We construct drawings of $K_n$ in the projective plane that disprove this.
\end{abstract}

\maketitle

\section{Introduction}

Hill's conjecture~\cite{HararyHill} states that the crossing number $\ucr{K_n}$ of the complete graph $K_n$ in the plane (equivalently, the sphere) is $\frac{1}{4}\floor{\frac{n}{2}}\floor{\frac{n-1}{2}}\floor{\frac{n-2}{2}}\floor{\frac{n-3}{2}}$. This has been verified only for $n\le 12$~\cite{Guy10,PanRichter}.

An elementary counting argument shows that $\bigl\{\ucr{K_n}/{n^4}\bigr\}_{n=1}^\infty$ is a non-decreasing sequence. {It is easy to see that a drawing of $K_n$ in which the vertices are in convex position and each edge is a straight segment has $\binom{n}{4}$ crossings, and so $\ucr{K_n}\le \binom{n}{4}$. Combining these two observations it follows that the} {\em plane constant\,} $c_{\real^2}:=\lim_{n\to\infty} \ucr{K_n}/n^4$ exists. The asymptotic version of Hill's conjecture then reads
\[
c_{\real^2}\stackrel{?}{=} \frac{1}{64}.
\]

In an interesting development, Moon~\cite{moon} proved that the expected number of crossings in a spherical drawing in which the points are randomly distributed and joined by geodesics is $\frac{n^4}{64}+O(n^3)$. Inspired by this exact match of the leading term (the coefficient of $n^4$) in Hill's conjectured value, it makes sense to consider models of random drawings of the complete graph in other surfaces. In this paper we concern ourselves with the crossing number $\pcr{K_n}$ of $K_n$ in the projective plane $P$.

\subsection{Drawings of $K_n$ in the projective plane}
Recently, Elkies~\cite{elkies} described a natural extension of Moon's model in the projective plane. We quote: {\sl ``Place $v$ points $\tilp_i$ randomly on the sphere $S$, and let $p_i$ be the corresponding points on $P$. Connect each pair of points $p_i,p_j$ with a shortest path, which is the image on $P$ of the great-circle arc on $S$ joining $\tilp_i$ to either $\tilp_j$ or $-\tilp_j$ [the antipodal point of $\tilp_j$], whichever is shorter. We show [\ldots] that any two such paths that do not share an endpoint cross with probability $1/\pi^2$.''}

The same argument as for the plane crossing number $\ucr{K_n}$ of $K_n$ shows that the {\em projective plane constant} $c_P:=\lim_{n\to\infty} \pcr{K_n}/n^4$ exists. Elkies's natural random model for drawing $K_n$ in the projective plane in particular implies that
\begin{equation}\label{eq:elk1}
c_P\le \frac{1}{8\pi^2},
\end{equation}
and he asked if ${1}/{8\pi^2}$ might be the actual value of $c_P$. (We note that Elkies investigated the number $\lim_{n\to\infty} 8\pcr{K_n}/n^4=8 c_P$, and asked if $8c_P=1/\pi^2$.) 

 Keeping in mind Moon's success of finding in his random model the same leading coefficient $1/64$ as in Hill's conjecture, indeed it makes sense to wonder:
\begin{equation}\label{eq:elkies}
c_P\stackrel{?}{=} \frac{1}{8\pi^2}.
\end{equation}

Elkies's result \eqref{eq:elk1} improves the upper bound $c_P<13/1028$, obtained by Koman~\cite{koman}. As Elkies points out, his bound is about $0.25\%$ better than Koman's bound.

\subsection{Our main result}

Using a refinement of Koman's construction, we show that equality does not hold in~\eqref{eq:elkies}:

\begin{theorem}\label{thm:main}
\[
c_P <  \frac{1}{8\pi^2}.
\]
\end{theorem}

The relevance of this result is not merely because we obtain a better upper bound for $c_P$, as our improvement over Elkies's bound is rather marginal. Its importance is that it shows that the natural extension of Moon's random model to the projective plane does not yield (asymptotically) optimal drawings of $K_n$ in the projective plane.

\subsection{Structure of the rest of this paper} In Section~\ref{sec:thedrawings} we describe, for each integer $n$ such that $n\equiv 2\pmod{8}$, and parameters $\alpha,\beta \in [0,2]$, a drawing $D(\alpha,\beta,n)$ of $K_n$ in the projective plane. (As we will see below, there are restrictions on $\alpha$ and $\beta$ that depend on $n$.) 

The proof of Theorem~\ref{thm:main} is in Section~\ref{sec:proofmain}. We start by giving the expression for the crossing number $\pcr{D(\alpha,\beta,n)}$ of $D(\alpha,\beta,n)$, for $\alpha,\beta\in[0,2]$ and $n\equiv 2\pmod{8}$. The calculations leading to this expression are totally elementary, if somewhat tedious. To avoid an interruption of the main discussion, these calculations are deferred to Section~\ref{sec:calculations}. To prove Theorem~\ref{thm:main}, we simply give explicit values $\alpha_0$ of $\alpha$ and $\beta_0$ of $\beta$ such that $\pcr{D(\alpha_0,\beta_0,n)}=c\cdot n^4 + O(n^3)$, for some $c < 1/8\pi^2$.

As we mentioned above, our drawings are a refinement of Koman's drawings in~\cite{koman}. In Section~\ref{sec:remarks} we explain why we could tell right away from Koman's calculations that his construction could be adapted to yield drawings with fewer crossings. Also in that section we describe further generalizations of our own construction.

\section{The drawings $D(\alpha,\beta,n)$ of $K_n$ in the projective plane}\label{sec:thedrawings}

In this section we describe a construction that yields drawings of $K_n$ in the projective plane, for $n\equiv 2 \pmod 8$. This is a generalization of Koman's construction in~\cite{koman}, and depends on two parameters $\alpha,\beta$ with $\alpha,\beta\in[0,2]$. (Koman's construction is obtained by setting $\alpha=\beta=1$.) As in~\cite{koman}, these drawings are obtained by first constructing a plane drawing, and then suitably identifying the sides of a polygon in this plane drawing. We start by describing this plane drawing, the {\em auxiliary model}.

\subsection{The auxiliary model $A(\alpha,\beta,n)$}

For the rest of the paper, $n\ge 10$ is a positive integer such that $n\equiv 2\pmod{8}$, and $k$ is the integer such that $n=8k+2$.

We construct the auxiliary model $A(\alpha,\beta,n)$ as follows. As illustrated in Figure~\ref{fig:440}, we start with three concentric circles $W,V$, and $U$. Place $4k+1$ vertices (the $w$-{\em vertices}) $w_0,\ldots,w_{4k}$ evenly distributed in the circle $W$, in this clockwise order, where $w_0$ is the highest point of $W$. Then place $4k+1$ vertices $v_0,\ldots,v_{4k}$ (the $v$-{\em vertices}) evenly distributed in the circle $V$, in this clockwise order, where $v_0$ is the lowest point of $V$. Finally, place $4k+1$ vertices $u_0,\ldots,u_{4k}$ (the $u$-{\em vertices}) evenly distributed in the circle $U$, placed in this clockwise order, where $u_0$ is the lowest point of $U$.

\begin{figure}[ht!]
\centering
\scalebox{0.45}{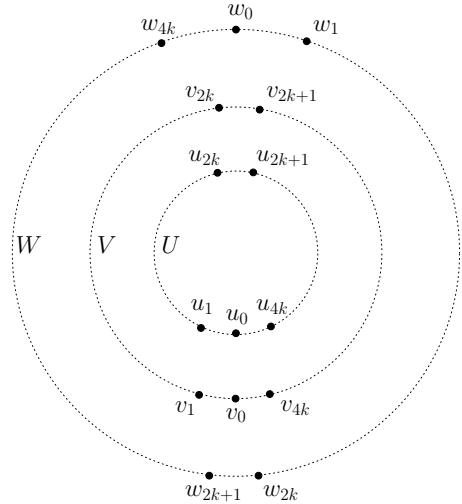}
\caption{The layout of the vertices $w_i,v_i$, and $u_i$ for $i=0,\ldots,4k$ in the auxiliary model $A(\alpha,\beta,n)$.}
\label{fig:440}
\end{figure}

Now let $\alpha, \beta \in [0,2]$ be numbers such that $\alpha k$ and $\beta k$ are integers. (That is, $\alpha (n-2)/8$ and $\beta (n-2)/8$ are integers.) We construct the model plane drawing $A(\alpha,\beta,n)$ by joining some pairs of vertices in $\{w_0,\ldots,w_{4k},v_0,\ldots,v_{4k},$ $u_0,\ldots,$ $u_{4k}\}$ with edges, which we colour for easy reference. The reader may wish to take a look at Figure~\ref{fig:150} to get a flavour of the final product of the construction for the particular case $n=10$ and $\alpha=\beta=1$. 

We use $WV$ to denote the annulus bounded by $W$ and $V$. Similarly, $VU$ is the annulus bounded by $V$ and $U$, and $WU$ is the annulus bounded by $W$ and $U$.

We join pairs of vertices in $\{w_0,\ldots,w_{4k},v_0,\ldots,v_{4k},u_0,\ldots,u_{4k}\}$ with edges that we colour green, red, brown, blue, and black, according to the following rules. We emphasize that indices are read modulo $4k+1$. We refer the reader to Figure~\ref{fig:450}.

\begin{figure}[ht!]
\centering
\scalebox{0.8}{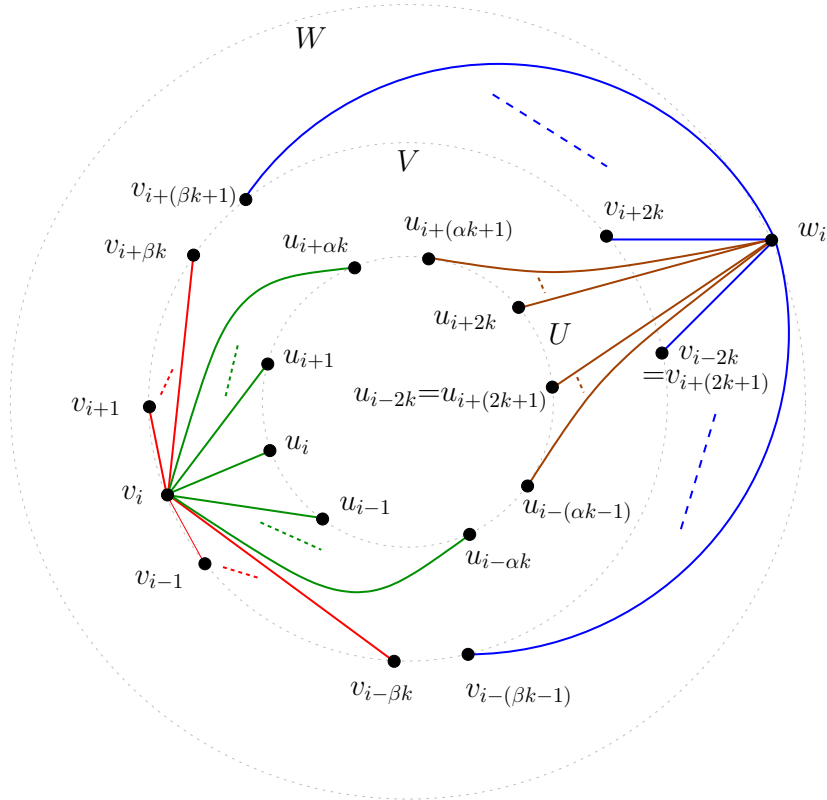}
\caption{The green and red edges incident with $v_i$, and the brown and blue edges incident with $w_i$.}
\label{fig:450}
\end{figure}

\medskip
\noindent{\bf The green edges.} For each $i=0,\ldots,4k$, we join the vertex $v_i$ with the $2\alpha k +1$ $u$-vertices that are closest to $v_i$. That is, we join $v_i$ with $u_{i- \alpha k}, \ldots, u_{i- 1}, u_i, u_{i+ 1}, \ldots, u_{i+ \alpha k}$. These edges are green, and are drawn inside $VU$. 


\medskip
\noindent{\bf The red edges.} For each $i=0,\ldots,4k$, we join the vertex $v_i$ with the $2\beta k$ $v$-vertices that are closest to $v_i$. That is, we join $v_i$ with $v_{i- \beta k}$, $\ldots, v_{i- 1}, v_{i+ 1}, \ldots$, $ v_{i+ \beta k}$. These edges are red, and they are drawn inside $VU$. 

\medskip
\noindent{\bf The brown edges.} For each $i=0,\ldots,4k$, we join the vertex $w_i$ with the $2(2-\alpha) k$ $u$-vertices that are closest to $w_i$. That is, we join $w_i$ with $u_{i- (\alpha k -1)},\ldots,u_{i- 2k}{=},u_{i+ (2k+1)}, u_{i+ 2k},\ldots, u_{i+ (\alpha k +1)}$. These edges are brown, and they are drawn inside $WU$ so that they cross the straight segment that joins $v_{i+ 2k}$ and $v_{i+ (2k+1)}$.

\medskip
\noindent{\bf The blue edges.} For each $i=0,\ldots,4k$, we join the vertex $w_i$ with the $2(2-\beta) k$ $v$-vertices that are closest to $w_i$. That is, we join $w_i$ with $v_{i-(\beta k -1)},\ldots,v_{i-2k}{=}v_{i+(2k+1)},  v_{i+ 2k},\ldots,v_{i+( \beta k +1)}$. These edges are blue, and they are drawn inside $WV$. 

\medskip
\noindent{\bf The black edges.} For each $i,j=0,\ldots,4k$, $i\neq j$, we join the vertex $u_i$ with the vertex $u_j$ (say with a straight segment) inside the disk bounded by $U$.
\medskip

In Figures~\ref{fig:160} and~\ref{fig:175} we give illustrations for some specific values of $n,\alpha$, and $\beta$.

\begin{figure}[ht!]
\centering
\scalebox{0.52}{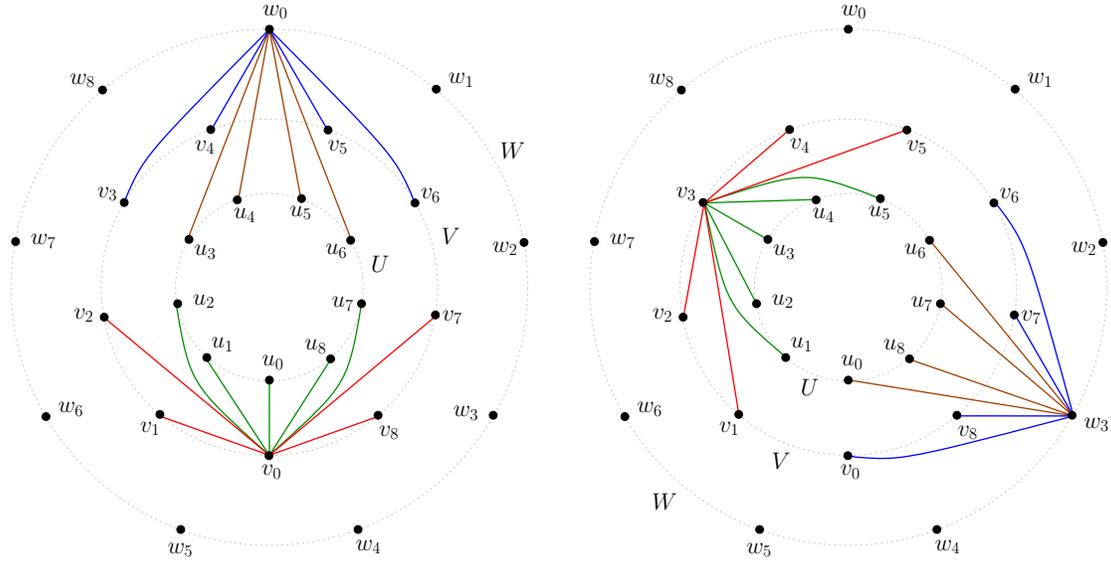}
\caption{On the left hand side we illustrate the red and green edges incident with $v_0$, and the brown and blue edges incident with $w_0$, for the case $n=18$ (that is, $k=2$) and $\alpha=\beta=1$. On the right hand side we illustrate the red and green edges incident with $v_3$, and the brown and blue edges incident with $w_3$.}
\label{fig:160}
\end{figure}

\begin{figure}[ht!]
\centering
\scalebox{0.52}{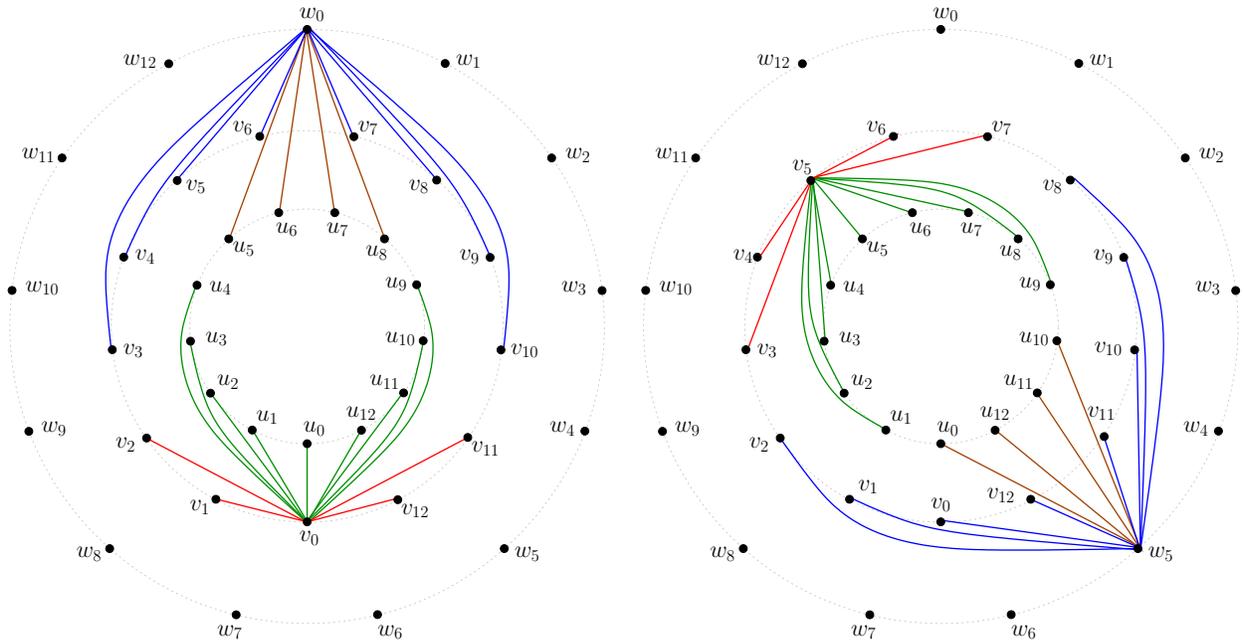}
\caption{On the left hand side we illustrate the red and green edges incident with $v_0$, and the brown and blue edges incident with $w_0$, for the case $n=26$ (that is, $k=3$) and $\alpha=4/3$ and $\beta=2/3$. On the right hand side we illustrate the red and green edges incident with $v_5$, and the brown and blue edges incident with $w_5$.}
\label{fig:175}
\end{figure}

As a full example, in Figure~\ref{fig:150} we illustrate the model plane drawing $A(1,1,18)$, that is, the case in which $\alpha=\beta=1$ and $n=18$, and so $k=2$.

\begin{figure}[ht!]
\centering
\scalebox{0.8}{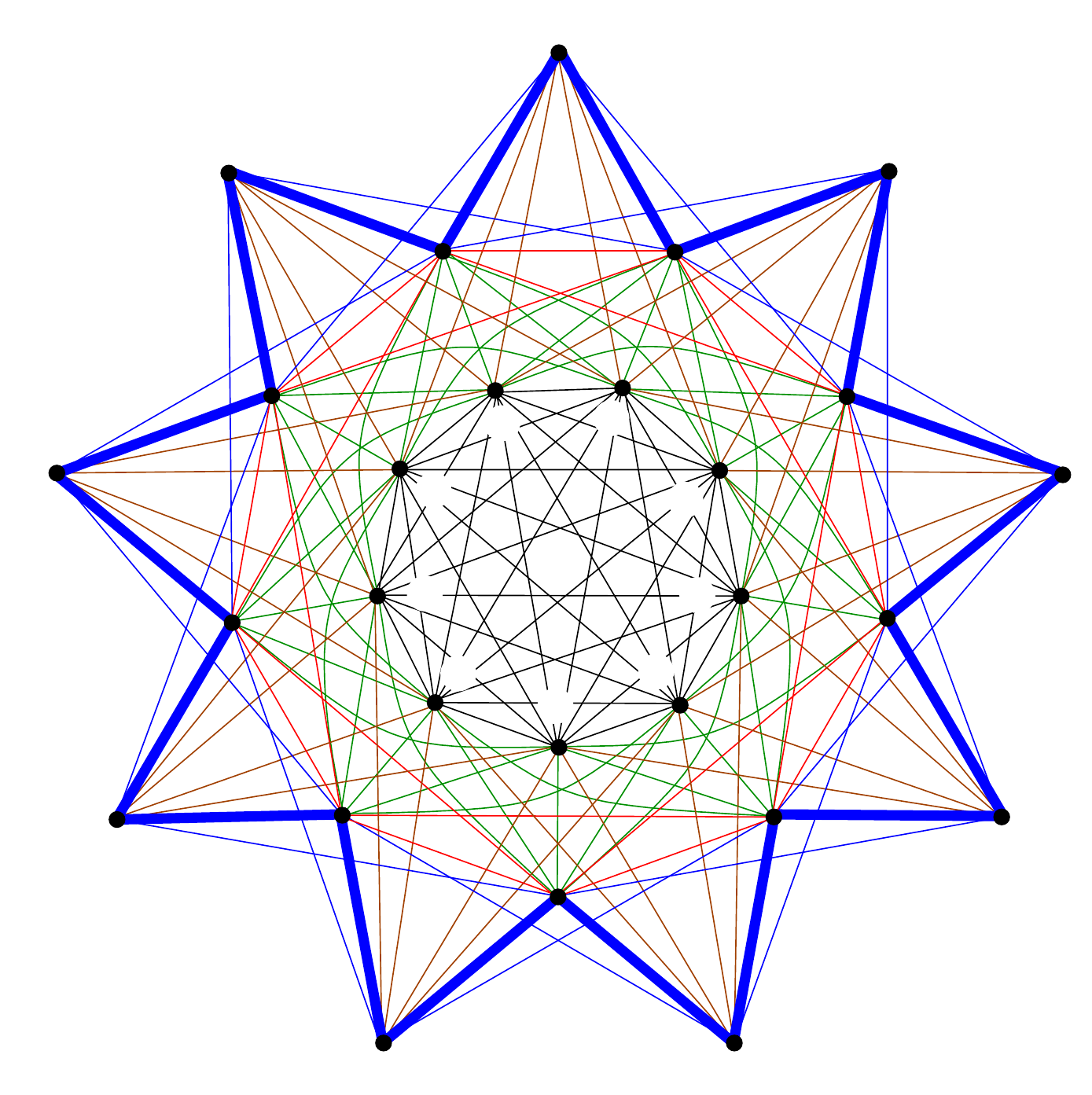}
\caption{The auxiliary model drawing $A(1,1,18)$. Each pair of opposite edges and opposite vertices of the thick blue polygon are identified to obtain the drawing $D(1,1,18)$ of $K_{18}$ in the projective plane (the part outside this polygon is discarded).}
\label{fig:150}
\end{figure}

\subsection{The drawing $D(\alpha,\beta,n)$ of $K_n$ in the projective plane}

From the model plane drawing $A(\alpha,\beta,n)$ we obtain a drawing $D(\alpha,\beta,n)$ of $K_n$ in the projective plane as follows. First, we discard the part of $A(\alpha,\beta,n)$ outside the polygon ${P{=}}w_0 v_{2k+1} w_1 v_{2k+2} \cdots w_{4k} v_{2k}$. In the example in Figure~\ref{fig:150} this is the thick blue polygon $w_0 v_5 w_1 v_6 w_2 v_7 w_3 v_8 w_4 v_0 w_5 v_1 w_6 v_2 w_7 v_3 w_8 v_4 w_0$.  

Then we identify each pair of opposite edges and each pair of opposite vertices in the polygon $P$. That is, we identify the vertices $w_i$ and $v_i$ {into a vertex $(vw)_i$} for $i=0,\ldots,4k$, and also identify each edge $w_i v_j$ of $P$ with the edge $w_j v_i$. Thus the pairs of edges that get identified are the pair $(w_i v_{i+ 2k},v_i w_{i+ 2k})$ and the pair $(w_i v_{i+(2k+1)},v_i w_{i+(2k+1)})$, for $i=0,\ldots,4k$. In the example in Figure~\ref{fig:150}, for instance, we identify the edges $w_0 v_5$ and $w_5 v_0$.

It is worth noting that these identified edges contain blue-blue crossings (crossings of a blue edge with a blue edge), and so the blue edges in the model drawing $A(\alpha,\beta,n)$ must be drawn so that the identification of these crossings is possible.

{The result of this process is a drawing $D(\alpha,\beta,n)$ in the projective plane. It is not difficult to verify that $D(\alpha,\beta,n)$ is a drawing of $K_n$. First of all, $D(\alpha,\beta,n)$ has $n=8k+2$ vertices, namely the {\em $u$-vertices} $u_0,\ldots,u_{4k}$ and the {\em $vw$-vertices} $(vw)_0,\ldots,(vw)_{4k}$. Each pair of $u$-vertices is joined by a black edge, and each $u$-vertex is joined to each $vw$-vertex either by a green or by a brown edge. Finally, each pair of $vw$-vertices are joined either by a red edge or by a blue edge, where the blue edge may be the union of two segments of blue edges back in $A(\alpha,\beta,n)$.}

\section{Proof of Theorem~\ref{thm:main}}\label{sec:proofmain}

To prove Theorem~\ref{thm:main} we simply exhibit numbers $\alpha_0$ and $\beta_0$ such that the number $\pcr{D(\alpha_0,\beta_0,n)}$ of crossings in $D(\alpha_0,\beta_0,n)$ is $c\cdot n^4+O(n^3)$, for some $c < 1/8\pi^2$.

Calculating $\pcr{D(\alpha,\beta,n)}$ is a routine exercise, if somewhat tedious. In order to continue with the proof of Theorem~\ref{thm:main} we just give this number (asymptotically)  in the following statement, and devote the next section to the elementary calculations leading to it.

For the next claim, we recall that the drawing $D(\alpha,\beta,n)$ is defined for integers $n$ such that $n=8k+2$ for some integer $k$, and for those $\alpha,\beta\in[0,2]$ such that $\alpha k$ and $\beta k$ are integers, that is, $\alpha (n-2)/8$ and $\beta (n-2)/8$ are integers.

\begin{claim}\label{cla:number}
Let $n\equiv 2\pmod{8}$, and let $\alpha,\beta\in [0,2]$ be such that $\alpha(n-2)/8$ and $\beta(n-2)/8$ are integers. Then 
\[
\pcr{D(\alpha,\beta,n)} =f(\alpha,\beta)\cdot n^4 + O(n^3),
\]
where
{\small
\begin{align*}
f(\alpha,\beta):=\frac{1}{8^4}\cdot \biggl( 
& \frac{32}{3} + 
\frac{4\beta^3}{3} +  
{8(2-\alpha)(2-\beta)^2 + 
4\alpha \beta^2} + 
\frac{32}{3} - \frac{32}{3}(\alpha-1)^3 + 4(2-\alpha)\beta^2 + \frac{16}{3}\alpha^3 +
\\
 &   \frac{16}{3}(2-\alpha)^3 + \frac{8}{3}(2-\beta)^3
\biggr).
\end{align*}
}
\end{claim}

We remark that the expression $f(\alpha,\beta)$  clearly can be simplified, but we chose to present it in this way to allow for an easier verification: these terms match in the given order the terms in which we will later express the exact value of $\pcr{D(\alpha,\beta,n)}$ (see proof of Claim~\ref{cla:number} in the next section).

\begin{proof}[Proof of Theorem~\ref{thm:main}]

Let $n$ be an integer such that $n\equiv 2\pmod{80}$. (Thus in particular $n\equiv 2\pmod{8}$). Note that for each such $n$ we have that $1.1\cdot (n-2)/8$ is an integer, and obviously also $1 \cdot (n-2)/8$ is an integer, and so the drawing $D(1.1,1,n)$ exists. 

By Claim~\ref{cla:number},
\[
\pcr{D(1.1,1,n)}= f(1.1,1)\cdot n^4 + O(n^3)\, { < 0.0126 \cdot n^4} + O(n^3).
\]

An elementary calculation shows that $\bigl\{\pcr{K_n}/{n^4}\bigr\}_{n=1}^\infty$ is a non-decreasing sequence. Thus the previous expression implies that $c_P < 0.0126$. Since $0.0126 < 1/8\pi^2 \approx 0.012665$, it follows that $c_P < 1/8\pi^2$.
\end{proof}

\section{Proof of Claim~\ref{cla:number}: the number of crossings in $D(\alpha,\beta,n)$}\label{sec:calculations}

The crossings in $A(\alpha,\beta,n)$ (and hence the crossings in $D(\alpha,\beta,n)$) are of nine {\em types}: green-green, red-red, brown-brown, blue-blue, black-black, red-green, red-brown, blue-brown, and green-brown. Obtaining the number of crossings in $D(\alpha,\beta,n)$ from the crossings in $A(\alpha,\beta,n)$ is quite easy, in view of the following remark, which is an immediate consequence of the way we obtain $D(\alpha,\beta,n)$ from $A(\alpha,\beta,n)$.

\begin{remark}\label{rem:rem1}
The number of each type of crossings is the same in $D(\alpha,\beta,n)$ as in $A(\alpha,\beta,n)$, with the exception of the blue-blue crossings: the number of blue-blue crossings in $D(\alpha,\beta,n)$ is half the number of blue-blue crossings in $A(\alpha,\beta,n)$.
\end{remark}

To identify each pair of opposite edges and vertices on the polygon, the order of the edge crossings on the opposite edges must be the same. It follows that each blue-blue crossing outside the polygon $P$ near the vertices $w_{i-1},v_{i+2k},w_i$ can also be found inside the polygon $P$ near the vertices $v_{i-1}, w_{i+2k}, v_i$. Similarly, each blue-blue crossing on the polygon will happen twice, once at each edge in a pair of opposite edges. Therefore, the number of blue-blue crossings in $D(\alpha,\beta,n)$ is half the number of blue-blue crossings in $A(\alpha,\beta,n)$. 

Recall that $k$ is the integer such that $n=8k+2$. We have the following statements, whose proofs are deferred for the moment. (We shall only give the proofs of (A)--(G), as the proofs of (H) and (I) are totally analogous to the proof of (G).)

\medskip
\noindent (A) {\sl The number of black-black crossings in $A(\alpha,\beta,n)$ is $\binom{4k+1}{4}$.}
\medskip

\medskip
\noindent (B) {\sl The number of red-red  crossings in $A(\alpha,\beta,n)$ is $(1/4)(4k+1)\binom{2\beta k}3$.}
\medskip

\medskip
\noindent (C) {\sl The number of blue-brown crossings in $A(\alpha,\beta,n)$ is $4k(4k+1)(2-\alpha)\binom{(2-\beta)k}2$.}
\medskip

\medskip
\noindent (D) {\sl The number of red-green crossings in $A(\alpha,\beta,n)$ is $(4k+1)(2\alpha k+1)\binom{\beta k}2$.}
\medskip

\medskip
\noindent (E) {\sl The number of green-brown crossings in $A(\alpha,\beta,n)$ is $2(4k+1)\left(\binom{2k+1}3-\binom{2(\alpha-1)k+1}3\right)$.}
\medskip

\medskip
\noindent (F) {\sl The number of red-brown crossings in $A(\alpha,\beta,n)$ is ${2(4k+1)}(2-\alpha)\binom{\beta k + 1}{2}$.}
\medskip

\medskip
\noindent (G) {\sl The number of green-green crossings in $A(\alpha,\beta,n)$ is $(4k+1) \binom{2\alpha k+1}3$.}
\medskip

\medskip
\noindent (H) {\sl The number of brown-brown crossings in $A(\alpha,\beta,n)$ is $(4k+1) \binom{2(2-\alpha) k}{3}$.}
\medskip

\medskip
\noindent (I) {\sl The number of blue-blue crossings in $A(\alpha,\beta,n)$ is $(4k+1) \binom{2(2-\beta)k}{3}$.}
\medskip

\begin{proof}[Proof of Claim~\ref{cla:number}]
In view of Remark~\ref{rem:rem1}, to obtain the number of crossings in $D(\alpha,\beta,n)$ it suffices to add the numbers in (A)--(H) plus half the number in (I). Thus

{\small 
\begin{align*}
&\pcr{D(\alpha,\beta,n})= 
\binom{4k+1}{4}+
(1/4)(4k+1)\binom{2\beta k}3+
4k(4k+1)(2-\alpha)\binom{2k-\beta k}2
\\ &+
(4k+1)(2\alpha k+1)\binom{\beta k}2+
2(4k+1)\left(\binom{2k+1}3-\binom{2(\alpha-1)k+1}3\right)+
2(4k+1)(2-\alpha)k\binom{\beta k + 1}{2}
\\ &+
(4k+1) \binom{2\alpha k+1}3+
(4k+1) \binom{2(2-\alpha) k}{3}+
(1/2)(4k+1) \binom{2(2-\beta)k}{3}.
\end{align*}
}

Recalling that $k=\frac{n-2}{8}$, a straightforward manipulation of this expression shows that the coefficient of $n^4$ in $\pcr{D(\alpha,\beta,n)}$ is $f(\alpha,\beta)$ in Claim~\ref{cla:number}.
\end{proof}

We conclude the section by proving statements (A)--(G). In the proofs, we say that a vertex $x$ is {\em incident} with a crossing if it is incident with one of the edges involved in the crossing. {With the exception of (A), the proofs of these statements involve counting (certain) crossings incident with either $v_0$ or $w_0$. For the reader's convenience, in Figure~\ref{fig:800} we illustrate some useful information on the edges incident with each of $v_0$ and $w_0$.}

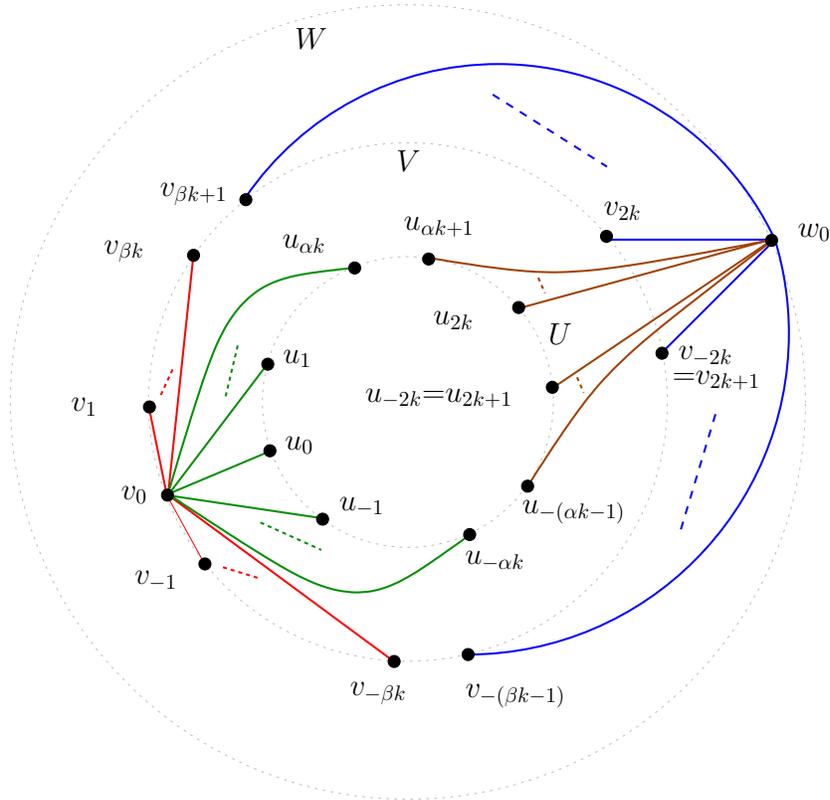
\begin{figure}[ht!]
\centering
\scalebox{0.8}{\input{800-1.pdf_t}}
\caption{The green and red edges incident with $v_0$, and the brown and blue edges incident with $w_0$.}
\label{fig:800}
\end{figure}

\begin{proof}[Proof of (A)]
The total number of black-black crossings is $\binom{4k+1}{4}$, because there are $4k+1$ $u$-vertices, and every set of four $u$-vertices defines a black-black crossing.
\end{proof}

\begin{proof}[Proof of (B)]
We show that the number $X$ of crossings of red edges incident with $v_0$ by red edges is $\binom{2\beta k}3$. Since there are $4k+1$ $v$-vertices, and each red-red crossing involves four $v$-vertices, by symmetry this implies that the total number of red-red crossings in $A(\alpha,\beta,n)$ is $(1/4)(4k+1)\binom{2\beta k}3$, as claimed.

Let $h<j<c$ be from $\{1,2,\dots,2\beta k\}$.  Half these triples have $j\le \beta k$. For this half, we will associate a red-red crossing involving $v_0 v_j$ and an edge incident with $v_h$. Let $i$ be the $(c-j)^{{\textrm{th}}}$ element of the list $h-(\beta k+1),h-(\beta k+2),\dots,4k,j+1,j+2,\dots,h+ \beta k$.  In this way, the triple $(h,j,c)$ corresponds precisely to the crossing $v_0v_j$ with $v_hv_i$.  

If $j>\beta k$, then we set $h'=4k+1-c$, $j'=4k+1-j$, and $c'=4k+1-h$ to get the symmetric situation on the other side of $v_0$.  Then $v_0v_{j'}$ crosses $v_{h'}v_{i'}$, with $i'$ the $(c'-j')^{{\textrm{th}}}$ entry in the list $h'- (\beta k+1), \dots, j'- 1, 1,2,\dots,h'+\beta k$.  Therefore $X=\binom{2\beta k}3$, as claimed.
\end{proof}

\begin{proof}[Proof of (C)]
We show that the number of blue-brown crossings that involve a brown edge incident with $w_0$ is $4k(2-\alpha)\binom{(2-\beta)k}2$. Since there are $4k+1$ $w$-vertices, by symmetry this implies that the total number of blue-brown crossings in $A(\alpha,\beta,n)$ is $4k(4k+1)(2-\alpha)\binom{(2-\beta)k}2$, as claimed.

A blue-brown crossing of the brown edge $w_0u_j$ with the blue edge $w_hv_i$ occurs when either {(i)} $1\le h\le 2k-\beta k -1$ and $h+(\beta k+ 1)\le i\le 2k$, or {(ii)} $2k+\beta k +2\le h\le 4k$ and $h+ (2k+(2-\beta) k)\ge i\ge 2k+1$.  Notice that $j$ is irrelevant, so that each blue edge crossing one brown edge incident with $w_0$ crosses all $2(2-\alpha)k$ brown edges incident with $w_0$.  Thus, the problem boils down to counting the number of pairs $(h,i)$ {that satisfy (i) or (ii). We show that the number of pairs $Z$ that satisfy (i) is $\binom{(2-\beta)k}{2}$. Since the number of pairs that satisfy (ii) is also $Z$, in view of the previous observation it will follow that the number of blue-brown crossings that involve a brown edge incident with $w_0$ is $2Z\bigl( 2(2-\alpha)k\bigr) = 4k(2-\alpha)\binom{(2-\beta)k}2$, as claimed.} 


For each pair $h'<i$ from $\{\beta k+1,\beta k+2,\dots,2k\}$ we set $h=h'-\beta k$.  Clearly, $1\le h\le 2k-\beta k -1$ and, since $h'<i\le 2k$, $h+\beta k<i\le 2k$.  As $h+\beta k = h+2k-(2-\beta)k$, this is precisely the constraint on $i$ {in (i). Thus $Z=\binom{2k-\beta k}{2} = \binom{(2-\beta)k}{2}$, as claimed}. 
\end{proof}

\begin{proof}[Proof of (D)]
We show that the number $X$ of crossings of green edges incident with $v_0$ by red edges is $(2\alpha k+1)\binom{\beta k}2$. Since there are $4k+1$ $v$-vertices, by symmetry this implies that the total number of red-green crossings in $A(\alpha,\beta,n)$ is $(4k+1)(2\alpha k+1)\binom{\beta k}2$, as claimed.

We count the crossings of the green edge $v_0$ to $u_j$ with red edges $v_hv_i$.  As in the blue-brown case, a red edge crosses either all $2\alpha k+1$ or none of the green edges incident with $v_0$.  We choose the labelling of each pair $(i,h)$ of indices so that $h\in\{1,2,\dots,\beta k-1\}$ and $h- \beta k \le i\le 4k$.

For each pair $i'<h'$ chosen from $\{1,2,\dots,\beta k\}$, set $h'=\beta k+1-h$ and $i=4k+1-i'$.  Then $i=4k+1-i' > 4k+1-h' = (4k+1)- (\beta k+1 -  h) = h-(\beta k+1)$, as required.  It follows that there are $\binom{\beta k}2$ pairs $(i,h)$, and so $X=(2\alpha k+1)\binom{\beta k}2$, as claimed.
\end{proof}

\begin{proof}[Proof of (E)]
We show that the number of green-brown crossings that involve a brown edge incident with $w_0$ is $2\left(\binom{2k+1}3-\binom{2(\alpha-1)k+1}3\right)$. Since there are $4k+1$ $w$-vertices, by symmetry this implies that the total number of green-brown crossings in $A(\alpha,\beta,n)$ is $2(4k+1)\left(\binom{2k+1}3-\binom{2(\alpha-1)k+1}3\right)$, as claimed.

Suppose $\alpha\ge 1$. Our aim is to compute {the} number of green-brown crossings that involve a brown edge incident with $w_0$. Half of these crossings involve $w_0u_i$ and $v_hu_j$ for which 
\[h\in {\{2,\dots,}2k\},\] 
\[i\in \{2k-(2-\alpha)k+1,2k-(2-\alpha)k+2,\dots,2k,2k+1,\dots,2k+(2-\alpha)k\}\,,\] and
\[j\in \{i+1,i+2\dots,h,h+1,\dots,h+\alpha k\}\,.\]

That is, $i\in\{\alpha k+1,\alpha k+2,\dots, (4-\alpha)k\}$ and $j\in \{i+1,i+2,\dots,h+\alpha k\}$.  We encode the possibilities $(i,j,h)$ by triples $i<j<c$ from $\{\alpha k+1,\alpha k+2,\dots, \alpha k + 2k+1\}$.  The $i$ and $j$ give the $u_i$ and $u_j$, while $c$ yields $h$ by the formula $h=c-\alpha k -1$.  Obviously, there are $\binom{2k+1}3$ triples. 

For any $j$ and $c$ with $j<c$, we have that $\alpha+2\le j\le c-1 = (h+\alpha k+1) -1=h+\alpha k$.  The upper bound $j\le h+\alpha k$ is precisely the upper limit on the green neighbours of $v_h$.  On the other hand, $j\ge \alpha k+2$ and $h\le 2k$.  Since $v_{2k}$ has $u_{2k-\alpha k}$ as a neighbour and $\alpha\ge 1$, $(2-\alpha)k<\alpha k+2$, so all the possibilities for $j$ and $c$ yield a green edge.

The brown edge $w_0u_i$ requires $i\in \{\alpha k+1,\dots,(4-\alpha)k\}$, so the triples with $i>(4-\alpha)k$ do not count.  There are $\binom{(\alpha k+2k+1-(4-\alpha)k}3$ triples with $i>(4-\alpha)k$.  Thus, the number of brown-green crossings incident with $w_0$ is $2\left(\binom{2k+1}3-\binom{2(\alpha-1)k+1}3\right)$, as claimed. A very similar argument applies to $\alpha<1$, starting with $v_0$, $w_h$, $u_i$, $u_j$.
\end{proof}

\begin{proof}[Proof of (F)]
We show that the number of red-brown crossings incident with $w_0$ is $2k(2{-}\alpha)$ $\binom{\beta k+1}{2}$. Since there are $4k+1$ $w$-vertices, and each red-brown crossing involves exactly one $w$-vertex, by symmetry it follows that the number of red-brown crossings in $A(\alpha,\beta,n)$ is $2k(4k+1)(2-\alpha)\binom{\beta k+1}{2}$, as claimed.

Each red-brown crossing incident with $w_0$ involves $v_i$ and $v_j$ with $i\in \{2k, 2k-1, \dots,2k-\beta k+1\}$ and $j\in \{2k+1,2k+2,\dots,i+\beta k\}$.  Each such $v_iv_j$ crosses all the brown edges incident with $w_0$.

Let $i'<j'$ be from $\{0,1,2\dots,\beta k\}$.   We choose $i=2k-i'$ and $j=i+j'=2k-i'+j'$.  Then $i\in \{2k,2k-1,\dots,2k-\beta k+1\}$ and $2k+1\le j\le i+\beta k$.  
 
Since $w_0$ has degree $2(2-\alpha)k$, there are $2(2-\alpha)k\binom{\beta k + 1}2$ brown-red crossings with brown edge incident with $w_0$, as claimed.  
\end{proof}

\begin{proof}[Proof of (G)]
We show that the number $X$ of green-green crossings incident with $v_0$ is $2\binom{2\alpha k+1}3$. Since there are $4k+1$ $v$-vertices, and each green-green crossing involves two $v$-vertices, by symmetry it follows that the number of green-green crossings in $A(\alpha,\beta,n)$ is $(1/2)(4k+1)X = (4k+1) \binom{2\alpha k+1}3$, as claimed.

To calculate $X$, we show that the number $Y$ of green-green crossings of the form $v_0u_j$ with $v_hu_i$ that have $h\in \{1,2,\dots,2k\}$ is $\binom{2\alpha k+1}3$. Since clearly $Y=X/2$, this implies that $X=2\binom{2\alpha k+1}3$.

In the case $\alpha\le 1$, there are no green-green crossings of the form $v_hu_i$ for $i\ge h$ and $h\ge \alpha k$.  (This is different in the case $\alpha> 1$, which we treat below).  The crossing $v_0u_j$ with $v_hu_i$ occurs if and only if $h\in \{1,2,\dots,2\alpha k-1\}$, $i\in \{h-\alpha k,h-\alpha k+1,\dots, \alpha k-1\}$, and $j\in \{i+1,\dots,\alpha k\}$.    Such a triple $(h,i,j)$ corresponds to the triple $(h,i+\alpha k+1,j+\alpha k+1)$ (so $h<i<j$) and this latter triple is from $\{1,2,\dots,2\alpha k+1\}$.  Thus, there are $\binom{2\alpha k+1}3$ triples, and so $Y=\binom{2\alpha k+1}3$, as claimed.

For $\alpha>1$, the situation is somewhat different. Recall that the goal is to calculate the number $Y$ of green-green crossings of the form $v_0u_j$ with $v_hu_i$ that have $h\in \{1,2,\dots,2k\}$. For $h+\alpha k \le (4-\alpha)k+1$, there are only the crossings as identified in the $\alpha\le 1$ case; these also occur up to $h=2k$.  However, for $h+\alpha k> (4-\alpha)k+1$, we also have the crossings $v_0u_j$ with $v_hu_i$ for $i\in\{h+\alpha k,h+\alpha k-1,\dots,(4-\alpha)k\}$ and $j\in\{i-1,i-2,\dots,(4-\alpha)k+1\}$.

The former apply for all $h=1,2,\dots,2k$, while the latter apply only for $h=2(2-\alpha)k+2,2(2-\alpha)k+3,\dots,2k$.  
For the former triples $(h,j,i)$, set $h'=2\alpha k+2-h$, $j'=\alpha k-j$, and $i'=\alpha k-i$.   We see that $1\le h\le 2k$ implies $2\alpha k+1\ge h' \ge 2(\alpha-1)k+2$.  Also $h -\alpha k \le i \le \alpha k-1$ implies $1\le i'\le \alpha k -h+\alpha k  = 2\alpha k -h = (2\alpha k - h)  + (2-2) = h'-2$, as expected.
Similarly, $h-\alpha k+1 \le j \le \alpha k$ implies $2\le j'\le h'-1$, so the triple $(h,i,j)$ corresponds to a triple $i'<j'<h'$ in $\{1,2,\dots,2\alpha k+1\}$ with $h' \ge 2(\alpha-1)k+2$.

For the remaining triples $(h,j,i)$, $2(2-\alpha)k+2\le h\le 2k$,   $4k+2-\alpha k\le i\le h+\alpha k$, and $4k+1-\alpha k\le j<i$.      We set  $h'=h+1-2(2-\alpha)k$, $i'=i-(4k-\alpha k)$, and $j'=j-(4k-\alpha k)$.   Then $2(2-\alpha)k+2\le h\le 2k$ translates to $ 3\le h'\le 2(\alpha-1)k+1$.  For $i'$, $4k+2-\alpha k \le i\le h+\alpha k$ translates to $ 2\le i' \le h+\alpha k - (4k- \alpha k) = h-2(2-\alpha)k= h'-1$.  Finally, for $j'$, $4k+1-\alpha k\le j< i $ translates to $1\le j'< i'$.  That is, the remaining triples $(h,i,j)$ correspond precisely with the triples $j'<i'<h'$ from $\{1,2\dots,2\alpha k+1\}$ having $h'\le 2(\alpha-1)k+1$.  Thus $Y=\binom{2\alpha k +1}3$, as claimed.
\end{proof}

\ignore{
{\small 
\begin{align*}
&\pcr{D(\alpha,\beta,n})=
(4k+1) \binom{2\alpha k+1}{3} +
\frac{4k+1}{2} \binom{2(2-\beta)k}{3} + (4k+1) \binom{2(2-\alpha) k}{3} + \frac{4k+1}{4} \binom{2\beta k}{3} \\ &+ 
\binom{4k+1}{4} +
(4k+1)(2\alpha k+1)\binom{\beta k}{2} + 2k(4k+1)(2-\alpha)\binom{\beta k + 1}{2} + 4k(4k+1)(2-\alpha) \binom{(2-\beta)k}{2}
 \nonumber \\ &+
2(4k+1)\left(\binom{2k+1}3-\binom{2(\alpha-1)k+1}3\right).
\end{align*}
}
}

\section{Concluding remarks}\label{sec:remarks}

We finally explain why a glance at Koman's calculations revealed to us that his constructions could be adapted to yield drawings with fewer crossings, and we explore further refinements of our own constructions and their impact on obtaining better estimates for $c_P$.

\subsection{The responsibility of vertices in crossing-minimal drawings of complete graphs} 

We recall that the {\em responsibility} of a vertex $v$ in a drawing $D$ is the number of crossings on edges incident to $v$. Many results on crossing numbers of complete graphs and complete bipartite graphs rely crucially on observations about the responsibility of vertices in crossing-minimal drawings. See for instance~\cite{Christian,guycrn,Kleitman,DanBruce}, and Sections 1.2 and 1.3 in~\cite{MarcusBook}.

Essentially the same arguments that show that the plane constant $c_{\real^2}$ and the projective plane constant $c_P$ exist, show that for each surface $\Sigma$ the constant $c_\Sigma:=\lim_{n\to \infty} cr_\Sigma(K_n)/n^4$ exists. There is a simple proof that the following holds for each surface $\Sigma$:

\begin{observation}\label{obs:rem1}
For each $n$, let $r_n$ be the responsibility of any vertex in a crossing-minimal drawing of $K_n$ in $\Sigma$.  Then $c_\Sigma=\lim_{n\to\infty}r_n/4 n^3$.
\end{observation}

Thus, roughly speaking, {\em all vertices in a crossing-minimal drawing of $K_n$ have essentially the same responsibility}. 

As we have mentioned, Koman's drawings~\cite{koman} are a particular instance of our drawings $D(\alpha,\beta,n)$: they are obtained by setting $\alpha=\beta=1$. In this case, as reported in~\cite[Eq.~(23)]{koman}, half the vertices have responsibility $\frac{19}{384}n^3+O(n^2)$, and the other half have responsibility $\frac{20}{384}n^3+O(n^2)$. Thus Koman's construction implies that $\pcr{K_n} \le (1/4)\bigl((n/2)(\frac{19}{384}n^3+O(n^2)) + (n/2)(\frac{20}{384}n^3 + O(n^2))\bigr) = \frac{39}{3072}n^4 + O(n^3)$, and so $c_P \le \frac{39}{3072}$.

In general, if we have drawings of $K_n$ in a surface $\Sigma$ in which half the vertices have responsibility $p\cdot n^3 + O(n^2)$, and the other half have responsibility $q\cdot n^3 + O(n^2)$, it easily follows that $c_\Sigma \le (1/8)(p+q)$. However, the following simpler variant of Observation~\ref{obs:rem1}, which is also not difficult to prove, claims that this last inequality must be strict if $p\neq q$. 

\begin{observation}\label{obs:rem2}
Let $\Sigma$ be a surface. Suppose that for each positive integer $n$ there is a drawing of $K_n$ in $\Sigma$ in which half the vertices have responsibility $p\cdot n^3 + O(n^2)$, and the other half have responsibility $q\cdot n^3 + O(n^2)$, where $p$ and $q$ are constants, $p\neq q$. Then $c_\Sigma < (1/8)(p+q)$.
\end{observation}

From this observation it follows that Koman's drawings are not asymptotically optimal. Elkies's recent result motivated us to refine Koman's construction and work out the corresponding calculations.

\subsection{Further generalizations of Koman's construction}
To show that the projective plane constant $c_P$ is strictly smaller than $1/8\pi^2$, we proved the inequality $c_P < 0.0126$. Our aim was only to show that $c_P < 1/8\pi^2$, but we could have easily obtained a better upper bound than $c_P < 0.0126$. Indeed, it is not difficult to see that
\[
c_P \le \min_{\alpha,\beta\in[0,2]} \,\, f(\alpha,\beta),
\]
where $f(\alpha,\beta)$ is as in Claim~\ref{cla:number}. Our bound $c_P < 0.0126$ follows since $f(1.1,1) < 0.0126$, but one can easily find values $\alpha_0$ and $\beta_0$ such that $f(\alpha_0,\beta_0)<f(1.1,1)$. 

We decided against pushing this further for three reasons. First, as we have mentioned, our chief goal was to show that $c_P < 1/8\pi^2$. Second, in order to obtain better estimates for $\min_{\alpha,\beta\in[0,2]} \,\, f(\alpha,\beta)$ one must use numerical tools, as $f(\alpha,\beta)$ is a polynomial of degree three in two variables, and there does not seem to be a closed expression in terms of $\alpha$ and $\beta$ in which $\min_{\alpha,\beta\in[0,2]} \,\, f(\alpha,\beta)$ is attained. 

The third reason, and the most important one, is that we {\em know} that $c_P$ is strictly smaller than $\min_{\alpha,\beta\in[0,2]} \,\, f(\alpha,\beta)$. Indeed, our refinement of Koman's construction admits a further generalization that yields even better estimates for $c_P$. One can prescribe that the number of $w$-vertices (which is the number of $v$-vertices, as they get identified to obtain $D(\alpha,\beta,n)$) in $A(\alpha,\beta,n)$ is not necessarily the same as the number of $u$-vertices.

Such a generalized construction involves three parameters: besides $\alpha$ and $\beta$, the additional parameter is the ratio between the number of $w$-vertices and the number of $u$-vertices. We worked out the calculations to obtain the number of crossings in these drawings, and verified that for suitable choices of these parameters they have strictly fewer crossings than $D(\alpha,\beta,n)$.

For some time we entertained the possibility of presenting this more general construction, but we ended up realizing this was not a good idea, for several reasons. First of all, the description of the construction itself is slightly more complicated than the description of the drawings $D(\alpha,\beta,n)$. The calculations to obtain the number of crossings in these drawings, although still totally elementary, naturally are also longer. Moreover, the improvement on $c_P$ is very marginal (we could show, for instance, that $c_P < 0.012547$).

Most importantly, all these considerations would have been overcome, and we would have chosen to present the full generalization of Koman's construction, if we had any reason to believe that these constructions yield drawings of $K_n$ in the projective plane that are (at least asymptotically) crossing-minimal, but we have no such reasons. However, we do find it intriguing that these constructions yield estimates of $c_P$ that are still very close to $1/8\pi^2$: they are not even $1\%$ better. Can one come up with drawings of $K_n$ in the projective plane that show that $c_P$ is substantially smaller than $1/8\pi^2$?

\section*{Acknowledgements}

{We thank two reviewers for their corrections and suggestions on the original version of this paper.} This project has received funding from NSERC Grant 50503-10940-500 and from the European Union's Horizon 2020 research and innovation programme under the Marie Sk\l{}odowska-Curie grant agreement No 754411, IST, Klosterneuburg, Austria.

\section*{Data Availability Statement}

Data sharing not applicable to this article as no datasets were generated or analysed during the current study.

\bibliographystyle{abbrv}
\bibliography{refs.bib}

\end{document}

%% file: 440-1.pdf_t
\begin{picture}(0,0)%
\includegraphics{440-1.pdf}%
\end{picture}%
\setlength{\unitlength}{4144sp}%
\begingroup\makeatletter\ifx\SetFigFont\undefined%
\gdef\SetFigFont#1#2#3#4#5{%
  \reset@font\fontsize{#1}{#2pt}%
  \fontfamily{#3}\fontseries{#4}\fontshape{#5}%
  \selectfont}%
\fi\endgroup%
\begin{picture}(5962,6604)(1520,-2465)
\put(3901, 79){\makebox(0,0)[lb]{\smash{{\SetFigFont{20}{24.0}{\familydefault}{\mddefault}{\updefault}{\color[rgb]{0,0,0}$u_1$}%
}}}}
\put(5091,-1271){\makebox(0,0)[lb]{\smash{{\SetFigFont{20}{24.0}{\familydefault}{\mddefault}{\updefault}{\color[rgb]{0,0,0}$v_{4k}$}%
}}}}
\put(4906,2909){\makebox(0,0)[lb]{\smash{{\SetFigFont{20}{24.0}{\familydefault}{\mddefault}{\updefault}{\color[rgb]{0,0,0}$v_{2k+1}$}%
}}}}
\put(4771,2054){\makebox(0,0)[lb]{\smash{{\SetFigFont{20}{24.0}{\familydefault}{\mddefault}{\updefault}{\color[rgb]{0,0,0}$u_{2k+1}$}%
}}}}
\put(4771, 74){\makebox(0,0)[lb]{\smash{{\SetFigFont{20}{24.0}{\familydefault}{\mddefault}{\updefault}{\color[rgb]{0,0,0}$u_{4k}$}%
}}}}
\put(3691,-1276){\makebox(0,0)[lb]{\smash{{\SetFigFont{20}{24.0}{\familydefault}{\mddefault}{\updefault}{\color[rgb]{0,0,0}$v_1$}%
}}}}
\put(4861,-2401){\makebox(0,0)[lb]{\smash{{\SetFigFont{20}{24.0}{\familydefault}{\mddefault}{\updefault}{\color[rgb]{0,0,0}$w_{2k}$}%
}}}}
\put(3826,-2401){\makebox(0,0)[lb]{\smash{{\SetFigFont{20}{24.0}{\familydefault}{\mddefault}{\updefault}{\color[rgb]{0,0,0}$w_{2k+1}$}%
}}}}
\put(3871,2054){\makebox(0,0)[lb]{\smash{{\SetFigFont{20}{24.0}{\familydefault}{\mddefault}{\updefault}{\color[rgb]{0,0,0}$u_{2k}$}%
}}}}
\put(3826,2909){\makebox(0,0)[lb]{\smash{{\SetFigFont{20}{24.0}{\familydefault}{\mddefault}{\updefault}{\color[rgb]{0,0,0}$v_{2k}$}%
}}}}
\put(4366,-61){\makebox(0,0)[lb]{\smash{{\SetFigFont{20}{24.0}{\familydefault}{\mddefault}{\updefault}{\color[rgb]{0,0,0}$u_0$}%
}}}}
\put(4366,-1366){\makebox(0,0)[lb]{\smash{{\SetFigFont{20}{24.0}{\familydefault}{\mddefault}{\updefault}{\color[rgb]{0,0,0}$v_0$}%
}}}}
\put(1576,839){\makebox(0,0)[lb]{\smash{{\SetFigFont{20}{24.0}{\familydefault}{\mddefault}{\updefault}{\color[rgb]{0,0,0}$W$}%
}}}}
\put(2656,839){\makebox(0,0)[lb]{\smash{{\SetFigFont{20}{24.0}{\familydefault}{\mddefault}{\updefault}{\color[rgb]{0,0,0}$V$}%
}}}}
\put(3511,839){\makebox(0,0)[lb]{\smash{{\SetFigFont{20}{24.0}{\familydefault}{\mddefault}{\updefault}{\color[rgb]{0,0,0}$U$}%
}}}}
\put(5526,3794){\makebox(0,0)[lb]{\smash{{\SetFigFont{20}{24.0}{\familydefault}{\mddefault}{\updefault}{\color[rgb]{0,0,0}$w_1$}%
}}}}
\put(3236,3804){\makebox(0,0)[lb]{\smash{{\SetFigFont{20}{24.0}{\familydefault}{\mddefault}{\updefault}{\color[rgb]{0,0,0}$w_{4k}$}%
}}}}
\put(4386,4004){\makebox(0,0)[lb]{\smash{{\SetFigFont{20}{24.0}{\familydefault}{\mddefault}{\updefault}{\color[rgb]{0,0,0}$w_0$}%
}}}}
\end{picture}%

%% file: 450-5.pdf_t
\begin{picture}(0,0)%
\includegraphics{450-5.pdf}%
\end{picture}%
\setlength{\unitlength}{4144sp}%
\begingroup\makeatletter\ifx\SetFigFont\undefined%
\gdef\SetFigFont#1#2#3#4#5{%
  \reset@font\fontsize{#1}{#2pt}%
  \fontfamily{#3}\fontseries{#4}\fontshape{#5}%
  \selectfont}%
\fi\endgroup%
\begin{picture}(5962,5962)(15473,-2145)
\put(16311,109){\makebox(0,0)[lb]{\smash{{\SetFigFont{14}{16.8}{\familydefault}{\mddefault}{\updefault}{\color[rgb]{0,0,0}$v_i$}%
}}}}
\put(17609,3481){\makebox(0,0)[lb]{\smash{{\SetFigFont{14}{16.8}{\familydefault}{\mddefault}{\updefault}{\color[rgb]{0,0,0}$W$}%
}}}}
\put(18369,2566){\makebox(0,0)[lb]{\smash{{\SetFigFont{14}{16.8}{\familydefault}{\mddefault}{\updefault}{\color[rgb]{0,0,0}$V$}%
}}}}
\put(19504,1271){\makebox(0,0)[lb]{\smash{{\SetFigFont{14}{16.8}{\familydefault}{\mddefault}{\updefault}{\color[rgb]{0,0,0}$U$}%
}}}}
\put(17534,481){\makebox(0,0)[lb]{\smash{{\SetFigFont{14}{16.8}{\familydefault}{\mddefault}{\updefault}{\color[rgb]{0,0,0}$u_i$}%
}}}}
\put(21369,2076){\makebox(0,0)[lb]{\smash{{\SetFigFont{14}{16.8}{\familydefault}{\mddefault}{\updefault}{\color[rgb]{0,0,0}$w_i$}%
}}}}
\put(16381,2374){\makebox(0,0)[lb]{\smash{{\SetFigFont{14}{16.8}{\familydefault}{\mddefault}{\updefault}{\color[rgb]{0,0,0}$\Va$}%
}}}}
\put(19919,2236){\makebox(0,0)[lb]{\smash{{\SetFigFont{14}{16.8}{\familydefault}{\mddefault}{\updefault}{\color[rgb]{0,0,0}$\Vb$}%
}}}}
\put(16181,1949){\makebox(0,0)[lb]{\smash{{\SetFigFont{14}{16.8}{\familydefault}{\mddefault}{\updefault}{\color[rgb]{0,0,0}$\Vc$}%
}}}}
\put(17516,1989){\makebox(0,0)[lb]{\smash{{\SetFigFont{14}{16.8}{\familydefault}{\mddefault}{\updefault}{\color[rgb]{0,0,0}$\Vd$}%
}}}}
\put(18424,2121){\makebox(0,0)[lb]{\smash{{\SetFigFont{14}{16.8}{\familydefault}{\mddefault}{\updefault}{\color[rgb]{0,0,0}$\Ve$}%
}}}}
\put(18644,1416){\makebox(0,0)[lb]{\smash{{\SetFigFont{14}{16.8}{\familydefault}{\mddefault}{\updefault}{\color[rgb]{0,0,0}$\Vf$}%
}}}}
\put(17521,1129){\makebox(0,0)[lb]{\smash{{\SetFigFont{14}{16.8}{\familydefault}{\mddefault}{\updefault}{\color[rgb]{0,0,0}$\Vg$}%
}}}}
\put(15936,769){\makebox(0,0)[lb]{\smash{{\SetFigFont{14}{16.8}{\familydefault}{\mddefault}{\updefault}{\color[rgb]{0,0,0}$\Vh$}%
}}}}
\put(20476,1154){\makebox(0,0)[lb]{\smash{{\SetFigFont{14}{16.8}{\familydefault}{\mddefault}{\updefault}{\color[rgb]{0,0,0}$\Vj$}%
}}}}
\put(17941, 29){\makebox(0,0)[lb]{\smash{{\SetFigFont{14}{16.8}{\familydefault}{\mddefault}{\updefault}{\color[rgb]{0,0,0}$\Vl$}%
}}}}
\put(19309,  6){\makebox(0,0)[lb]{\smash{{\SetFigFont{14}{16.8}{\familydefault}{\mddefault}{\updefault}{\color[rgb]{0,0,0}$\Vm$}%
}}}}
\put(16411,-526){\makebox(0,0)[lb]{\smash{{\SetFigFont{14}{16.8}{\familydefault}{\mddefault}{\updefault}{\color[rgb]{0,0,0}$\Vn$}%
}}}}
\put(18881,-376){\makebox(0,0)[lb]{\smash{{\SetFigFont{14}{16.8}{\familydefault}{\mddefault}{\updefault}{\color[rgb]{0,0,0}$\Vo$}%
}}}}
\put(18021,-1361){\makebox(0,0)[lb]{\smash{{\SetFigFont{14}{16.8}{\familydefault}{\mddefault}{\updefault}{\color[rgb]{0,0,0}$\Vp$}%
}}}}
\put(18884,-1369){\makebox(0,0)[lb]{\smash{{\SetFigFont{14}{16.8}{\familydefault}{\mddefault}{\updefault}{\color[rgb]{0,0,0}$\Vq$}%
}}}}
\put(18046,839){\makebox(0,0)[lb]{\smash{{\SetFigFont{14}{16.8}{\familydefault}{\mddefault}{\updefault}{\color[rgb]{0,0,0}$\Vi$}%
}}}}
\put(20206,974){\makebox(0,0)[lb]{\smash{{\SetFigFont{14}{16.8}{\familydefault}{\mddefault}{\updefault}{\color[rgb]{0,0,0}$\Vk$}%
}}}}
\end{picture}%

%% file: 160-2.pdf_t
\begin{picture}(0,0)%
\includegraphics{160-2.pdf}%
\end{picture}%
\setlength{\unitlength}{4144sp}%
\begingroup\makeatletter\ifx\SetFigFont\undefined%
\gdef\SetFigFont#1#2#3#4#5{%
  \reset@font\fontsize{#1}{#2pt}%
  \fontfamily{#3}\fontseries{#4}\fontshape{#5}%
  \selectfont}%
\fi\endgroup%
\begin{picture}(12624,6385)(1520,-2281)
\put(11396,1659){\makebox(0,0)[lb]{\smash{{\SetFigFont{17}{20.4}{\familydefault}{\mddefault}{\updefault}{\color[rgb]{0,0,0}$u_5$}%
}}}}
\put(11081,3969){\makebox(0,0)[lb]{\smash{{\SetFigFont{17}{20.4}{\familydefault}{\mddefault}{\updefault}{\color[rgb]{0,0,0}$w_0$}%
}}}}
\put(8866,3179){\makebox(0,0)[lb]{\smash{{\SetFigFont{17}{20.4}{\familydefault}{\mddefault}{\updefault}{\color[rgb]{0,0,0}$w_8$}%
}}}}
\put(12156,-2186){\makebox(0,0)[lb]{\smash{{\SetFigFont{17}{20.4}{\familydefault}{\mddefault}{\updefault}{\color[rgb]{0,0,0}$w_4$}%
}}}}
\put(13221,3184){\makebox(0,0)[lb]{\smash{{\SetFigFont{17}{20.4}{\familydefault}{\mddefault}{\updefault}{\color[rgb]{0,0,0}$w_1$}%
}}}}
\put(13151,454){\makebox(0,0)[lb]{\smash{{\SetFigFont{17}{20.4}{\familydefault}{\mddefault}{\updefault}{\color[rgb]{0,0,0}$v_7$}%
}}}}
\put(12901,1859){\makebox(0,0)[lb]{\smash{{\SetFigFont{17}{20.4}{\familydefault}{\mddefault}{\updefault}{\color[rgb]{0,0,0}$v_6$}%
}}}}
\put(10751,1629){\makebox(0,0)[lb]{\smash{{\SetFigFont{17}{20.4}{\familydefault}{\mddefault}{\updefault}{\color[rgb]{0,0,0}$u_4$}%
}}}}
\put(10296,1209){\makebox(0,0)[lb]{\smash{{\SetFigFont{17}{20.4}{\familydefault}{\mddefault}{\updefault}{\color[rgb]{0,0,0}$u_3$}%
}}}}
\put(10281,639){\makebox(0,0)[lb]{\smash{{\SetFigFont{17}{20.4}{\familydefault}{\mddefault}{\updefault}{\color[rgb]{0,0,0}$u_2$}%
}}}}
\put(10496,139){\makebox(0,0)[lb]{\smash{{\SetFigFont{17}{20.4}{\familydefault}{\mddefault}{\updefault}{\color[rgb]{0,0,0}$u_1$}%
}}}}
\put(11076,-66){\makebox(0,0)[lb]{\smash{{\SetFigFont{17}{20.4}{\familydefault}{\mddefault}{\updefault}{\color[rgb]{0,0,0}$u_0$}%
}}}}
\put(11806,1219){\makebox(0,0)[lb]{\smash{{\SetFigFont{17}{20.4}{\familydefault}{\mddefault}{\updefault}{\color[rgb]{0,0,0}$u_6$}%
}}}}
\put(11836,2414){\makebox(0,0)[lb]{\smash{{\SetFigFont{17}{20.4}{\familydefault}{\mddefault}{\updefault}{\color[rgb]{0,0,0}$v_5$}%
}}}}
\put(8911,479){\makebox(0,0)[lb]{\smash{{\SetFigFont{17}{20.4}{\familydefault}{\mddefault}{\updefault}{\color[rgb]{0,0,0}$v_2$}%
}}}}
\put(9676,-826){\makebox(0,0)[lb]{\smash{{\SetFigFont{17}{20.4}{\familydefault}{\mddefault}{\updefault}{\color[rgb]{0,0,0}$v_1$}%
}}}}
\put(12421,-871){\makebox(0,0)[lb]{\smash{{\SetFigFont{17}{20.4}{\familydefault}{\mddefault}{\updefault}{\color[rgb]{0,0,0}$v_8$}%
}}}}
\put(11881,614){\makebox(0,0)[lb]{\smash{{\SetFigFont{17}{20.4}{\familydefault}{\mddefault}{\updefault}{\color[rgb]{0,0,0}$u_7$}%
}}}}
\put(11071,-1321){\makebox(0,0)[lb]{\smash{{\SetFigFont{17}{20.4}{\familydefault}{\mddefault}{\updefault}{\color[rgb]{0,0,0}$v_0$}%
}}}}
\put(9991,-2221){\makebox(0,0)[lb]{\smash{{\SetFigFont{17}{20.4}{\familydefault}{\mddefault}{\updefault}{\color[rgb]{0,0,0}$w_5$}%
}}}}
\put(11566,119){\makebox(0,0)[lb]{\smash{{\SetFigFont{17}{20.4}{\familydefault}{\mddefault}{\updefault}{\color[rgb]{0,0,0}$u_8$}%
}}}}
\put(9181,1874){\makebox(0,0)[lb]{\smash{{\SetFigFont{17}{20.4}{\familydefault}{\mddefault}{\updefault}{\color[rgb]{0,0,0}$v_3$}%
}}}}
\put(8416,1334){\makebox(0,0)[lb]{\smash{{\SetFigFont{17}{20.4}{\familydefault}{\mddefault}{\updefault}{\color[rgb]{0,0,0}$w_7$}%
}}}}
\put(13726,1244){\makebox(0,0)[lb]{\smash{{\SetFigFont{17}{20.4}{\familydefault}{\mddefault}{\updefault}{\color[rgb]{0,0,0}$w_2$}%
}}}}
\put(8731,-601){\makebox(0,0)[lb]{\smash{{\SetFigFont{17}{20.4}{\familydefault}{\mddefault}{\updefault}{\color[rgb]{0,0,0}$w_6$}%
}}}}
\put(10491,2429){\makebox(0,0)[lb]{\smash{{\SetFigFont{17}{20.4}{\familydefault}{\mddefault}{\updefault}{\color[rgb]{0,0,0}$v_4$}%
}}}}
\put(13871,-781){\makebox(0,0)[lb]{\smash{{\SetFigFont{17}{20.4}{\familydefault}{\mddefault}{\updefault}{\color[rgb]{0,0,0}$w_3$}%
}}}}
\put(10611,-401){\makebox(0,0)[lb]{\smash{{\SetFigFont{17}{20.4}{\familydefault}{\mddefault}{\updefault}{\color[rgb]{0,0,0}$U$}%
}}}}
\put(10281,-1251){\makebox(0,0)[lb]{\smash{{\SetFigFont{17}{20.4}{\familydefault}{\mddefault}{\updefault}{\color[rgb]{0,0,0}$V$}%
}}}}
\put(8901,-1726){\makebox(0,0)[lb]{\smash{{\SetFigFont{17}{20.4}{\familydefault}{\mddefault}{\updefault}{\color[rgb]{0,0,0}$W$}%
}}}}
\put(4736,1659){\makebox(0,0)[lb]{\smash{{\SetFigFont{17}{20.4}{\familydefault}{\mddefault}{\updefault}{\color[rgb]{0,0,0}$u_5$}%
}}}}
\put(4421,3969){\makebox(0,0)[lb]{\smash{{\SetFigFont{17}{20.4}{\familydefault}{\mddefault}{\updefault}{\color[rgb]{0,0,0}$w_0$}%
}}}}
\put(2206,3179){\makebox(0,0)[lb]{\smash{{\SetFigFont{17}{20.4}{\familydefault}{\mddefault}{\updefault}{\color[rgb]{0,0,0}$w_8$}%
}}}}
\put(5496,-2186){\makebox(0,0)[lb]{\smash{{\SetFigFont{17}{20.4}{\familydefault}{\mddefault}{\updefault}{\color[rgb]{0,0,0}$w_4$}%
}}}}
\put(6561,3184){\makebox(0,0)[lb]{\smash{{\SetFigFont{17}{20.4}{\familydefault}{\mddefault}{\updefault}{\color[rgb]{0,0,0}$w_1$}%
}}}}
\put(6491,454){\makebox(0,0)[lb]{\smash{{\SetFigFont{17}{20.4}{\familydefault}{\mddefault}{\updefault}{\color[rgb]{0,0,0}$v_7$}%
}}}}
\put(6241,1859){\makebox(0,0)[lb]{\smash{{\SetFigFont{17}{20.4}{\familydefault}{\mddefault}{\updefault}{\color[rgb]{0,0,0}$v_6$}%
}}}}
\put(4091,1629){\makebox(0,0)[lb]{\smash{{\SetFigFont{17}{20.4}{\familydefault}{\mddefault}{\updefault}{\color[rgb]{0,0,0}$u_4$}%
}}}}
\put(3636,1209){\makebox(0,0)[lb]{\smash{{\SetFigFont{17}{20.4}{\familydefault}{\mddefault}{\updefault}{\color[rgb]{0,0,0}$u_3$}%
}}}}
\put(3621,639){\makebox(0,0)[lb]{\smash{{\SetFigFont{17}{20.4}{\familydefault}{\mddefault}{\updefault}{\color[rgb]{0,0,0}$u_2$}%
}}}}
\put(3836,139){\makebox(0,0)[lb]{\smash{{\SetFigFont{17}{20.4}{\familydefault}{\mddefault}{\updefault}{\color[rgb]{0,0,0}$u_1$}%
}}}}
\put(4416,-66){\makebox(0,0)[lb]{\smash{{\SetFigFont{17}{20.4}{\familydefault}{\mddefault}{\updefault}{\color[rgb]{0,0,0}$u_0$}%
}}}}
\put(5146,1219){\makebox(0,0)[lb]{\smash{{\SetFigFont{17}{20.4}{\familydefault}{\mddefault}{\updefault}{\color[rgb]{0,0,0}$u_6$}%
}}}}
\put(3646,2414){\makebox(0,0)[lb]{\smash{{\SetFigFont{17}{20.4}{\familydefault}{\mddefault}{\updefault}{\color[rgb]{0,0,0}$v_4$}%
}}}}
\put(5176,2414){\makebox(0,0)[lb]{\smash{{\SetFigFont{17}{20.4}{\familydefault}{\mddefault}{\updefault}{\color[rgb]{0,0,0}$v_5$}%
}}}}
\put(2251,479){\makebox(0,0)[lb]{\smash{{\SetFigFont{17}{20.4}{\familydefault}{\mddefault}{\updefault}{\color[rgb]{0,0,0}$v_2$}%
}}}}
\put(3016,-826){\makebox(0,0)[lb]{\smash{{\SetFigFont{17}{20.4}{\familydefault}{\mddefault}{\updefault}{\color[rgb]{0,0,0}$v_1$}%
}}}}
\put(5761,-871){\makebox(0,0)[lb]{\smash{{\SetFigFont{17}{20.4}{\familydefault}{\mddefault}{\updefault}{\color[rgb]{0,0,0}$v_8$}%
}}}}
\put(5221,614){\makebox(0,0)[lb]{\smash{{\SetFigFont{17}{20.4}{\familydefault}{\mddefault}{\updefault}{\color[rgb]{0,0,0}$u_7$}%
}}}}
\put(4411,-1321){\makebox(0,0)[lb]{\smash{{\SetFigFont{17}{20.4}{\familydefault}{\mddefault}{\updefault}{\color[rgb]{0,0,0}$v_0$}%
}}}}
\put(3331,-2221){\makebox(0,0)[lb]{\smash{{\SetFigFont{17}{20.4}{\familydefault}{\mddefault}{\updefault}{\color[rgb]{0,0,0}$w_5$}%
}}}}
\put(4906,119){\makebox(0,0)[lb]{\smash{{\SetFigFont{17}{20.4}{\familydefault}{\mddefault}{\updefault}{\color[rgb]{0,0,0}$u_8$}%
}}}}
\put(2521,1874){\makebox(0,0)[lb]{\smash{{\SetFigFont{17}{20.4}{\familydefault}{\mddefault}{\updefault}{\color[rgb]{0,0,0}$v_3$}%
}}}}
\put(5671,1019){\makebox(0,0)[lb]{\smash{{\SetFigFont{17}{20.4}{\familydefault}{\mddefault}{\updefault}{\color[rgb]{0,0,0}$U$}%
}}}}
\put(6481,1334){\makebox(0,0)[lb]{\smash{{\SetFigFont{17}{20.4}{\familydefault}{\mddefault}{\updefault}{\color[rgb]{0,0,0}$V$}%
}}}}
\put(7156,2324){\makebox(0,0)[lb]{\smash{{\SetFigFont{17}{20.4}{\familydefault}{\mddefault}{\updefault}{\color[rgb]{0,0,0}$W$}%
}}}}
\put(1756,1334){\makebox(0,0)[lb]{\smash{{\SetFigFont{17}{20.4}{\familydefault}{\mddefault}{\updefault}{\color[rgb]{0,0,0}$w_7$}%
}}}}
\put(7066,1244){\makebox(0,0)[lb]{\smash{{\SetFigFont{17}{20.4}{\familydefault}{\mddefault}{\updefault}{\color[rgb]{0,0,0}$w_2$}%
}}}}
\put(6616,-646){\makebox(0,0)[lb]{\smash{{\SetFigFont{17}{20.4}{\familydefault}{\mddefault}{\updefault}{\color[rgb]{0,0,0}$w_3$}%
}}}}
\put(2071,-601){\makebox(0,0)[lb]{\smash{{\SetFigFont{17}{20.4}{\familydefault}{\mddefault}{\updefault}{\color[rgb]{0,0,0}$w_6$}%
}}}}
\end{picture}%

%% file: 175-3.pdf_t
\begin{picture}(0,0)%
\includegraphics{175-3.pdf}%
\end{picture}%
\setlength{\unitlength}{4144sp}%
\begingroup\makeatletter\ifx\SetFigFont\undefined%
\gdef\SetFigFont#1#2#3#4#5{%
  \reset@font\fontsize{#1}{#2pt}%
  \fontfamily{#3}\fontseries{#4}\fontshape{#5}%
  \selectfont}%
\fi\endgroup%
\begin{picture}(14184,7305)(918,-7366)
\put(11566,-196){\makebox(0,0)[lb]{\smash{{\SetFigFont{17}{20.4}{\familydefault}{\mddefault}{\updefault}{\color[rgb]{0,0,0}$w_0$}%
}}}}
\put(14041,-6406){\makebox(0,0)[lb]{\smash{{\SetFigFont{17}{20.4}{\familydefault}{\mddefault}{\updefault}{\color[rgb]{0,0,0}$w_5$}%
}}}}
\put(10621,-7306){\makebox(0,0)[lb]{\smash{{\SetFigFont{17}{20.4}{\familydefault}{\mddefault}{\updefault}{\color[rgb]{0,0,0}$w_7$}%
}}}}
\put(12466,-7306){\makebox(0,0)[lb]{\smash{{\SetFigFont{17}{20.4}{\familydefault}{\mddefault}{\updefault}{\color[rgb]{0,0,0}$w_6$}%
}}}}
\put(9001,-6406){\makebox(0,0)[lb]{\smash{{\SetFigFont{17}{20.4}{\familydefault}{\mddefault}{\updefault}{\color[rgb]{0,0,0}$w_8$}%
}}}}
\put(8371,-1816){\makebox(0,0)[lb]{\smash{{\SetFigFont{17}{20.4}{\familydefault}{\mddefault}{\updefault}{\color[rgb]{0,0,0}$w_{11}$}%
}}}}
\put(8416,-3346){\makebox(0,0)[lb]{\smash{{\SetFigFont{17}{20.4}{\familydefault}{\mddefault}{\updefault}{\color[rgb]{0,0,0}$w_{10}$}%
}}}}
\put(8596,-4966){\makebox(0,0)[lb]{\smash{{\SetFigFont{17}{20.4}{\familydefault}{\mddefault}{\updefault}{\color[rgb]{0,0,0}$w_9$}%
}}}}
\put(13366,-736){\makebox(0,0)[lb]{\smash{{\SetFigFont{17}{20.4}{\familydefault}{\mddefault}{\updefault}{\color[rgb]{0,0,0}$w_1$}%
}}}}
\put(14626,-1816){\makebox(0,0)[lb]{\smash{{\SetFigFont{17}{20.4}{\familydefault}{\mddefault}{\updefault}{\color[rgb]{0,0,0}$w_2$}%
}}}}
\put(11611,-4966){\makebox(0,0)[lb]{\smash{{\SetFigFont{17}{20.4}{\familydefault}{\mddefault}{\updefault}{\color[rgb]{0,0,0}$u_0$}%
}}}}
\put(11071,-4831){\makebox(0,0)[lb]{\smash{{\SetFigFont{17}{20.4}{\familydefault}{\mddefault}{\updefault}{\color[rgb]{0,0,0}$u_1$}%
}}}}
\put(10621,-4426){\makebox(0,0)[lb]{\smash{{\SetFigFont{17}{20.4}{\familydefault}{\mddefault}{\updefault}{\color[rgb]{0,0,0}$u_2$}%
}}}}
\put(10486,-3931){\makebox(0,0)[lb]{\smash{{\SetFigFont{17}{20.4}{\familydefault}{\mddefault}{\updefault}{\color[rgb]{0,0,0}$u_3$}%
}}}}
\put(10531,-3346){\makebox(0,0)[lb]{\smash{{\SetFigFont{17}{20.4}{\familydefault}{\mddefault}{\updefault}{\color[rgb]{0,0,0}$u_4$}%
}}}}
\put(10801,-2896){\makebox(0,0)[lb]{\smash{{\SetFigFont{17}{20.4}{\familydefault}{\mddefault}{\updefault}{\color[rgb]{0,0,0}$u_5$}%
}}}}
\put(9541,-4066){\makebox(0,0)[lb]{\smash{{\SetFigFont{17}{20.4}{\familydefault}{\mddefault}{\updefault}{\color[rgb]{0,0,0}$v_3$}%
}}}}
\put(9541,-736){\makebox(0,0)[lb]{\smash{{\SetFigFont{17}{20.4}{\familydefault}{\mddefault}{\updefault}{\color[rgb]{0,0,0}$w_{12}$}%
}}}}
\put(12241,-1501){\makebox(0,0)[lb]{\smash{{\SetFigFont{17}{20.4}{\familydefault}{\mddefault}{\updefault}{\color[rgb]{0,0,0}$v_7$}%
}}}}
\put(12826,-2176){\makebox(0,0)[lb]{\smash{{\SetFigFont{17}{20.4}{\familydefault}{\mddefault}{\updefault}{\color[rgb]{0,0,0}$v_8$}%
}}}}
\put(13411,-2986){\makebox(0,0)[lb]{\smash{{\SetFigFont{17}{20.4}{\familydefault}{\mddefault}{\updefault}{\color[rgb]{0,0,0}$v_9$}%
}}}}
\put(9451,-5146){\makebox(0,0)[lb]{\smash{{\SetFigFont{17}{20.4}{\familydefault}{\mddefault}{\updefault}{\color[rgb]{0,0,0}$v_2$}%
}}}}
\put(11296,-2671){\makebox(0,0)[lb]{\smash{{\SetFigFont{17}{20.4}{\familydefault}{\mddefault}{\updefault}{\color[rgb]{0,0,0}$u_6$}%
}}}}
\put(11881,-2671){\makebox(0,0)[lb]{\smash{{\SetFigFont{17}{20.4}{\familydefault}{\mddefault}{\updefault}{\color[rgb]{0,0,0}$u_7$}%
}}}}
\put(12331,-2941){\makebox(0,0)[lb]{\smash{{\SetFigFont{17}{20.4}{\familydefault}{\mddefault}{\updefault}{\color[rgb]{0,0,0}$u_8$}%
}}}}
\put(12556,-3301){\makebox(0,0)[lb]{\smash{{\SetFigFont{17}{20.4}{\familydefault}{\mddefault}{\updefault}{\color[rgb]{0,0,0}$u_9$}%
}}}}
\put(12556,-3931){\makebox(0,0)[lb]{\smash{{\SetFigFont{17}{20.4}{\familydefault}{\mddefault}{\updefault}{\color[rgb]{0,0,0}$u_{10}$}%
}}}}
\put(12376,-4426){\makebox(0,0)[lb]{\smash{{\SetFigFont{17}{20.4}{\familydefault}{\mddefault}{\updefault}{\color[rgb]{0,0,0}$u_{11}$}%
}}}}
\put(12106,-4786){\makebox(0,0)[lb]{\smash{{\SetFigFont{17}{20.4}{\familydefault}{\mddefault}{\updefault}{\color[rgb]{0,0,0}$u_{12}$}%
}}}}
\put(10801,-1501){\makebox(0,0)[lb]{\smash{{\SetFigFont{17}{20.4}{\familydefault}{\mddefault}{\updefault}{\color[rgb]{0,0,0}$v_6$}%
}}}}
\put(14581,-3346){\makebox(0,0)[lb]{\smash{{\SetFigFont{17}{20.4}{\familydefault}{\mddefault}{\updefault}{\color[rgb]{0,0,0}$w_3$}%
}}}}
\put(14446,-4966){\makebox(0,0)[lb]{\smash{{\SetFigFont{17}{20.4}{\familydefault}{\mddefault}{\updefault}{\color[rgb]{0,0,0}$w_4$}%
}}}}
\put(4276,-196){\makebox(0,0)[lb]{\smash{{\SetFigFont{17}{20.4}{\familydefault}{\mddefault}{\updefault}{\color[rgb]{0,0,0}$w_0$}%
}}}}
\put(6751,-6406){\makebox(0,0)[lb]{\smash{{\SetFigFont{17}{20.4}{\familydefault}{\mddefault}{\updefault}{\color[rgb]{0,0,0}$w_5$}%
}}}}
\put(4276,-6226){\makebox(0,0)[lb]{\smash{{\SetFigFont{17}{20.4}{\familydefault}{\mddefault}{\updefault}{\color[rgb]{0,0,0}$v_0$}%
}}}}
\put(3331,-7306){\makebox(0,0)[lb]{\smash{{\SetFigFont{17}{20.4}{\familydefault}{\mddefault}{\updefault}{\color[rgb]{0,0,0}$w_7$}%
}}}}
\put(5176,-7306){\makebox(0,0)[lb]{\smash{{\SetFigFont{17}{20.4}{\familydefault}{\mddefault}{\updefault}{\color[rgb]{0,0,0}$w_6$}%
}}}}
\put(1711,-6406){\makebox(0,0)[lb]{\smash{{\SetFigFont{17}{20.4}{\familydefault}{\mddefault}{\updefault}{\color[rgb]{0,0,0}$w_8$}%
}}}}
\put(1081,-1816){\makebox(0,0)[lb]{\smash{{\SetFigFont{17}{20.4}{\familydefault}{\mddefault}{\updefault}{\color[rgb]{0,0,0}$w_{11}$}%
}}}}
\put(1126,-3346){\makebox(0,0)[lb]{\smash{{\SetFigFont{17}{20.4}{\familydefault}{\mddefault}{\updefault}{\color[rgb]{0,0,0}$w_{10}$}%
}}}}
\put(1306,-4966){\makebox(0,0)[lb]{\smash{{\SetFigFont{17}{20.4}{\familydefault}{\mddefault}{\updefault}{\color[rgb]{0,0,0}$w_9$}%
}}}}
\put(6076,-736){\makebox(0,0)[lb]{\smash{{\SetFigFont{17}{20.4}{\familydefault}{\mddefault}{\updefault}{\color[rgb]{0,0,0}$w_1$}%
}}}}
\put(7336,-1816){\makebox(0,0)[lb]{\smash{{\SetFigFont{17}{20.4}{\familydefault}{\mddefault}{\updefault}{\color[rgb]{0,0,0}$w_2$}%
}}}}
\put(4321,-4966){\makebox(0,0)[lb]{\smash{{\SetFigFont{17}{20.4}{\familydefault}{\mddefault}{\updefault}{\color[rgb]{0,0,0}$u_0$}%
}}}}
\put(3781,-4831){\makebox(0,0)[lb]{\smash{{\SetFigFont{17}{20.4}{\familydefault}{\mddefault}{\updefault}{\color[rgb]{0,0,0}$u_1$}%
}}}}
\put(3331,-4426){\makebox(0,0)[lb]{\smash{{\SetFigFont{17}{20.4}{\familydefault}{\mddefault}{\updefault}{\color[rgb]{0,0,0}$u_2$}%
}}}}
\put(3196,-3931){\makebox(0,0)[lb]{\smash{{\SetFigFont{17}{20.4}{\familydefault}{\mddefault}{\updefault}{\color[rgb]{0,0,0}$u_3$}%
}}}}
\put(3241,-3346){\makebox(0,0)[lb]{\smash{{\SetFigFont{17}{20.4}{\familydefault}{\mddefault}{\updefault}{\color[rgb]{0,0,0}$u_4$}%
}}}}
\put(3511,-2896){\makebox(0,0)[lb]{\smash{{\SetFigFont{17}{20.4}{\familydefault}{\mddefault}{\updefault}{\color[rgb]{0,0,0}$u_5$}%
}}}}
\put(5446,-5911){\makebox(0,0)[lb]{\smash{{\SetFigFont{17}{20.4}{\familydefault}{\mddefault}{\updefault}{\color[rgb]{0,0,0}$v_{12}$}%
}}}}
\put(6256,-5191){\makebox(0,0)[lb]{\smash{{\SetFigFont{17}{20.4}{\familydefault}{\mddefault}{\updefault}{\color[rgb]{0,0,0}$v_{11}$}%
}}}}
\put(6706,-4066){\makebox(0,0)[lb]{\smash{{\SetFigFont{17}{20.4}{\familydefault}{\mddefault}{\updefault}{\color[rgb]{0,0,0}$v_{10}$}%
}}}}
\put(2971,-2176){\makebox(0,0)[lb]{\smash{{\SetFigFont{17}{20.4}{\familydefault}{\mddefault}{\updefault}{\color[rgb]{0,0,0}$v_5$}%
}}}}
\put(2386,-2986){\makebox(0,0)[lb]{\smash{{\SetFigFont{17}{20.4}{\familydefault}{\mddefault}{\updefault}{\color[rgb]{0,0,0}$v_4$}%
}}}}
\put(2251,-4066){\makebox(0,0)[lb]{\smash{{\SetFigFont{17}{20.4}{\familydefault}{\mddefault}{\updefault}{\color[rgb]{0,0,0}$v_3$}%
}}}}
\put(2251,-736){\makebox(0,0)[lb]{\smash{{\SetFigFont{17}{20.4}{\familydefault}{\mddefault}{\updefault}{\color[rgb]{0,0,0}$w_{12}$}%
}}}}
\put(4951,-1501){\makebox(0,0)[lb]{\smash{{\SetFigFont{17}{20.4}{\familydefault}{\mddefault}{\updefault}{\color[rgb]{0,0,0}$v_7$}%
}}}}
\put(5536,-2176){\makebox(0,0)[lb]{\smash{{\SetFigFont{17}{20.4}{\familydefault}{\mddefault}{\updefault}{\color[rgb]{0,0,0}$v_8$}%
}}}}
\put(6121,-2986){\makebox(0,0)[lb]{\smash{{\SetFigFont{17}{20.4}{\familydefault}{\mddefault}{\updefault}{\color[rgb]{0,0,0}$v_9$}%
}}}}
\put(3016,-5866){\makebox(0,0)[lb]{\smash{{\SetFigFont{17}{20.4}{\familydefault}{\mddefault}{\updefault}{\color[rgb]{0,0,0}$v_1$}%
}}}}
\put(2161,-5146){\makebox(0,0)[lb]{\smash{{\SetFigFont{17}{20.4}{\familydefault}{\mddefault}{\updefault}{\color[rgb]{0,0,0}$v_2$}%
}}}}
\put(4006,-2671){\makebox(0,0)[lb]{\smash{{\SetFigFont{17}{20.4}{\familydefault}{\mddefault}{\updefault}{\color[rgb]{0,0,0}$u_6$}%
}}}}
\put(4591,-2671){\makebox(0,0)[lb]{\smash{{\SetFigFont{17}{20.4}{\familydefault}{\mddefault}{\updefault}{\color[rgb]{0,0,0}$u_7$}%
}}}}
\put(5041,-2941){\makebox(0,0)[lb]{\smash{{\SetFigFont{17}{20.4}{\familydefault}{\mddefault}{\updefault}{\color[rgb]{0,0,0}$u_8$}%
}}}}
\put(5266,-3301){\makebox(0,0)[lb]{\smash{{\SetFigFont{17}{20.4}{\familydefault}{\mddefault}{\updefault}{\color[rgb]{0,0,0}$u_9$}%
}}}}
\put(5266,-3931){\makebox(0,0)[lb]{\smash{{\SetFigFont{17}{20.4}{\familydefault}{\mddefault}{\updefault}{\color[rgb]{0,0,0}$u_{10}$}%
}}}}
\put(5086,-4426){\makebox(0,0)[lb]{\smash{{\SetFigFont{17}{20.4}{\familydefault}{\mddefault}{\updefault}{\color[rgb]{0,0,0}$u_{11}$}%
}}}}
\put(4816,-4786){\makebox(0,0)[lb]{\smash{{\SetFigFont{17}{20.4}{\familydefault}{\mddefault}{\updefault}{\color[rgb]{0,0,0}$u_{12}$}%
}}}}
\put(3511,-1501){\makebox(0,0)[lb]{\smash{{\SetFigFont{17}{20.4}{\familydefault}{\mddefault}{\updefault}{\color[rgb]{0,0,0}$v_6$}%
}}}}
\put(7291,-3346){\makebox(0,0)[lb]{\smash{{\SetFigFont{17}{20.4}{\familydefault}{\mddefault}{\updefault}{\color[rgb]{0,0,0}$w_3$}%
}}}}
\put(7156,-4966){\makebox(0,0)[lb]{\smash{{\SetFigFont{17}{20.4}{\familydefault}{\mddefault}{\updefault}{\color[rgb]{0,0,0}$w_4$}%
}}}}
\put(13421,-4041){\makebox(0,0)[lb]{\smash{{\SetFigFont{17}{20.4}{\familydefault}{\mddefault}{\updefault}{\color[rgb]{0,0,0}$v_{10}$}%
}}}}
\put(11541,-5841){\makebox(0,0)[lb]{\smash{{\SetFigFont{17}{20.4}{\familydefault}{\mddefault}{\updefault}{\color[rgb]{0,0,0}$v_0$}%
}}}}
\put(10706,-5636){\makebox(0,0)[lb]{\smash{{\SetFigFont{17}{20.4}{\familydefault}{\mddefault}{\updefault}{\color[rgb]{0,0,0}$v_1$}%
}}}}
\put(12186,-5731){\makebox(0,0)[lb]{\smash{{\SetFigFont{17}{20.4}{\familydefault}{\mddefault}{\updefault}{\color[rgb]{0,0,0}$v_{12}$}%
}}}}
\put(13341,-4786){\makebox(0,0)[lb]{\smash{{\SetFigFont{17}{20.4}{\familydefault}{\mddefault}{\updefault}{\color[rgb]{0,0,0}$v_{11}$}%
}}}}
\put(9266,-2981){\makebox(0,0)[lb]{\smash{{\SetFigFont{17}{20.4}{\familydefault}{\mddefault}{\updefault}{\color[rgb]{0,0,0}$v_4$}%
}}}}
\put(9946,-1956){\makebox(0,0)[lb]{\smash{{\SetFigFont{17}{20.4}{\familydefault}{\mddefault}{\updefault}{\color[rgb]{0,0,0}$v_5$}%
}}}}
\end{picture}%

%% file: 170-1.pdf_t
\begin{picture}(0,0)%
\includegraphics{170-1.pdf}%
\end{picture}%
\setlength{\unitlength}{4144sp}%
\begingroup\makeatletter\ifx\SetFigFont\undefined%
\gdef\SetFigFont#1#2#3#4#5{%
  \reset@font\fontsize{#1}{#2pt}%
  \fontfamily{#3}\fontseries{#4}\fontshape{#5}%
  \selectfont}%
\fi\endgroup%
\begin{picture}(6345,6355)(1251,-2251)
\put(4421,3969){\makebox(0,0)[lb]{\smash{{\SetFigFont{12}{14.4}{\familydefault}{\mddefault}{\updefault}{\color[rgb]{0,0,0}$w_0$}%
}}}}
\put(6561,3184){\makebox(0,0)[lb]{\smash{{\SetFigFont{12}{14.4}{\familydefault}{\mddefault}{\updefault}{\color[rgb]{0,0,0}$w_1$}%
}}}}
\put(7581,1339){\makebox(0,0)[lb]{\smash{{\SetFigFont{12}{14.4}{\familydefault}{\mddefault}{\updefault}{\color[rgb]{0,0,0}$w_2$}%
}}}}
\put(7186,-786){\makebox(0,0)[lb]{\smash{{\SetFigFont{12}{14.4}{\familydefault}{\mddefault}{\updefault}{\color[rgb]{0,0,0}$w_3$}%
}}}}
\put(5496,-2186){\makebox(0,0)[lb]{\smash{{\SetFigFont{12}{14.4}{\familydefault}{\mddefault}{\updefault}{\color[rgb]{0,0,0}$w_4$}%
}}}}
\put(3386,-2191){\makebox(0,0)[lb]{\smash{{\SetFigFont{12}{14.4}{\familydefault}{\mddefault}{\updefault}{\color[rgb]{0,0,0}$w_5$}%
}}}}
\put(1651,-791){\makebox(0,0)[lb]{\smash{{\SetFigFont{12}{14.4}{\familydefault}{\mddefault}{\updefault}{\color[rgb]{0,0,0}$w_6$}%
}}}}
\put(1266,1349){\makebox(0,0)[lb]{\smash{{\SetFigFont{12}{14.4}{\familydefault}{\mddefault}{\updefault}{\color[rgb]{0,0,0}$w_7$}%
}}}}
\put(2206,3179){\makebox(0,0)[lb]{\smash{{\SetFigFont{12}{14.4}{\familydefault}{\mddefault}{\updefault}{\color[rgb]{0,0,0}$w_8$}%
}}}}
\put(3056,-771){\makebox(0,0)[lb]{\smash{{\SetFigFont{12}{14.4}{\familydefault}{\mddefault}{\updefault}{\color[rgb]{0,0,0}$v_1$}%
}}}}
\put(5146,2774){\makebox(0,0)[lb]{\smash{{\SetFigFont{12}{14.4}{\familydefault}{\mddefault}{\updefault}{\color[rgb]{0,0,0}$v_5$}%
}}}}
\put(6241,1859){\makebox(0,0)[lb]{\smash{{\SetFigFont{12}{14.4}{\familydefault}{\mddefault}{\updefault}{\color[rgb]{0,0,0}$v_6$}%
}}}}
\put(6491,454){\makebox(0,0)[lb]{\smash{{\SetFigFont{12}{14.4}{\familydefault}{\mddefault}{\updefault}{\color[rgb]{0,0,0}$v_7$}%
}}}}
\put(2366,459){\makebox(0,0)[lb]{\smash{{\SetFigFont{12}{14.4}{\familydefault}{\mddefault}{\updefault}{\color[rgb]{0,0,0}$v_2$}%
}}}}
\put(4401,-1271){\makebox(0,0)[lb]{\smash{{\SetFigFont{12}{14.4}{\familydefault}{\mddefault}{\updefault}{\color[rgb]{0,0,0}$v_0$}%
}}}}
\put(5771,-766){\makebox(0,0)[lb]{\smash{{\SetFigFont{12}{14.4}{\familydefault}{\mddefault}{\updefault}{\color[rgb]{0,0,0}$v_8$}%
}}}}
\put(3621,639){\makebox(0,0)[lb]{\smash{{\SetFigFont{12}{14.4}{\familydefault}{\mddefault}{\updefault}{\color[rgb]{0,0,0}$u_2$}%
}}}}
\put(5141,1194){\makebox(0,0)[lb]{\smash{{\SetFigFont{12}{14.4}{\familydefault}{\mddefault}{\updefault}{\color[rgb]{0,0,0}$u_6$}%
}}}}
\put(5226,629){\makebox(0,0)[lb]{\smash{{\SetFigFont{12}{14.4}{\familydefault}{\mddefault}{\updefault}{\color[rgb]{0,0,0}$u_7$}%
}}}}
\put(4971,179){\makebox(0,0)[lb]{\smash{{\SetFigFont{12}{14.4}{\familydefault}{\mddefault}{\updefault}{\color[rgb]{0,0,0}$u_8$}%
}}}}
\put(4401,-21){\makebox(0,0)[lb]{\smash{{\SetFigFont{12}{14.4}{\familydefault}{\mddefault}{\updefault}{\color[rgb]{0,0,0}$u_0$}%
}}}}
\put(3836,179){\makebox(0,0)[lb]{\smash{{\SetFigFont{12}{14.4}{\familydefault}{\mddefault}{\updefault}{\color[rgb]{0,0,0}$u_1$}%
}}}}
\put(3671,1249){\makebox(0,0)[lb]{\smash{{\SetFigFont{12}{14.4}{\familydefault}{\mddefault}{\updefault}{\color[rgb]{0,0,0}$u_3$}%
}}}}
\put(4076,1629){\makebox(0,0)[lb]{\smash{{\SetFigFont{12}{14.4}{\familydefault}{\mddefault}{\updefault}{\color[rgb]{0,0,0}$u_4$}%
}}}}
\put(4711,1664){\makebox(0,0)[lb]{\smash{{\SetFigFont{12}{14.4}{\familydefault}{\mddefault}{\updefault}{\color[rgb]{0,0,0}$u_5$}%
}}}}
\put(3682,2759){\makebox(0,0)[lb]{\smash{{\SetFigFont{12}{14.4}{\familydefault}{\mddefault}{\updefault}{\color[rgb]{0,0,0}$v_4$}%
}}}}
\put(2621,1881){\makebox(0,0)[lb]{\smash{{\SetFigFont{12}{14.4}{\familydefault}{\mddefault}{\updefault}{\color[rgb]{0,0,0}$v_3$}%
}}}}
\end{picture}%

%% file: 800-1.pdf_t
\begin{picture}(0,0)%
\includegraphics{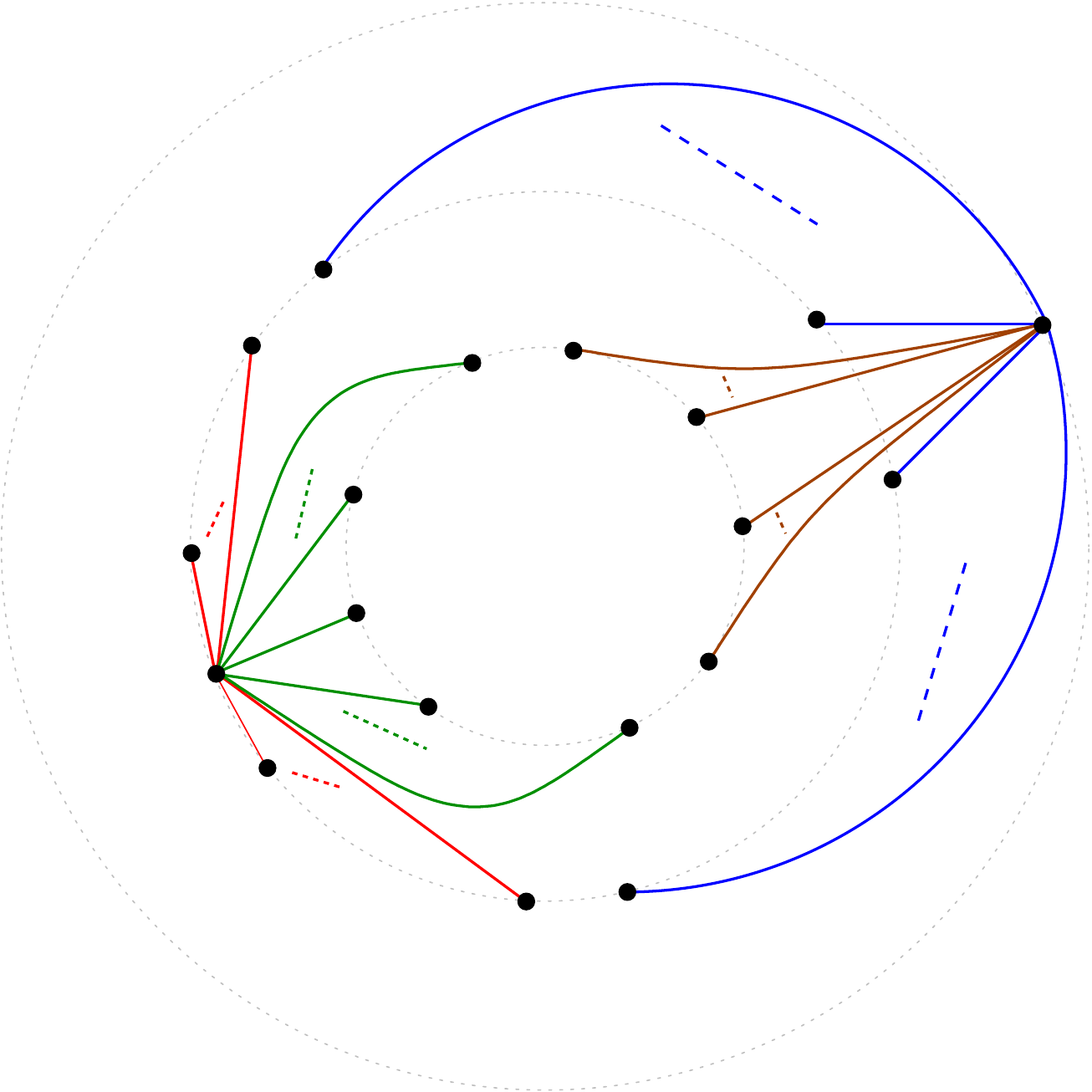}%
\end{picture}%
\setlength{\unitlength}{4144sp}%
\begingroup\makeatletter\ifx\SetFigFont\undefined%
\gdef\SetFigFont#1#2#3#4#5{%
  \reset@font\fontsize{#1}{#2pt}%
  \fontfamily{#3}\fontseries{#4}\fontshape{#5}%
  \selectfont}%
\fi\endgroup%
\begin{picture}(5962,5962)(15473,-2145)
\put(16311,109){\makebox(0,0)[lb]{\smash{{\SetFigFont{14}{16.8}{\familydefault}{\mddefault}{\updefault}{\color[rgb]{0,0,0}$v_0$}%
}}}}
\put(17609,3481){\makebox(0,0)[lb]{\smash{{\SetFigFont{14}{16.8}{\familydefault}{\mddefault}{\updefault}{\color[rgb]{0,0,0}$W$}%
}}}}
\put(18369,2566){\makebox(0,0)[lb]{\smash{{\SetFigFont{14}{16.8}{\familydefault}{\mddefault}{\updefault}{\color[rgb]{0,0,0}$V$}%
}}}}
\put(19504,1271){\makebox(0,0)[lb]{\smash{{\SetFigFont{14}{16.8}{\familydefault}{\mddefault}{\updefault}{\color[rgb]{0,0,0}$U$}%
}}}}
\put(20476,1154){\makebox(0,0)[lb]{\smash{{\SetFigFont{14}{16.8}{\familydefault}{\mddefault}{\updefault}{\color[rgb]{0,0,0}$v_{-2k}$}%
}}}}
\put(16600,2374){\makebox(0,0)[lb]{\smash{{\SetFigFont{14}{16.8}{\familydefault}{\mddefault}{\updefault}{\color[rgb]{0,0,0}$v_{\beta k +1}$}%
}}}}
\put(15936,769){\makebox(0,0)[lb]{\smash{{\SetFigFont{14}{16.8}{\familydefault}{\mddefault}{\updefault}{\color[rgb]{0,0,0}$v_{1}$}%
}}}}
\put(16411,-526){\makebox(0,0)[lb]{\smash{{\SetFigFont{14}{16.8}{\familydefault}{\mddefault}{\updefault}{\color[rgb]{0,0,0}$v_{-1}$}%
}}}}
\put(18884,-1369){\makebox(0,0)[lb]{\smash{{\SetFigFont{14}{16.8}{\familydefault}{\mddefault}{\updefault}{\color[rgb]{0,0,0}$v_{-(\beta k -1)}$}%
}}}}
\put(18021,-1361){\makebox(0,0)[lb]{\smash{{\SetFigFont{14}{16.8}{\familydefault}{\mddefault}{\updefault}{\color[rgb]{0,0,0}$v_{-\beta k}$}%
}}}}
\put(18881,-376){\makebox(0,0)[lb]{\smash{{\SetFigFont{14}{16.8}{\familydefault}{\mddefault}{\updefault}{\color[rgb]{0,0,0}$u_{-\alpha k}$}%
}}}}
\put(19309,  6){\makebox(0,0)[lb]{\smash{{\SetFigFont{14}{16.8}{\familydefault}{\mddefault}{\updefault}{\color[rgb]{0,0,0}$u_{ -(\alpha k -1)}$}%
}}}}
\put(17941, 29){\makebox(0,0)[lb]{\smash{{\SetFigFont{14}{16.8}{\familydefault}{\mddefault}{\updefault}{\color[rgb]{0,0,0}$u_{-1}$}%
}}}}
\put(17534,481){\makebox(0,0)[lb]{\smash{{\SetFigFont{14}{16.8}{\familydefault}{\mddefault}{\updefault}{\color[rgb]{0,0,0}$u_0$}%
}}}}
\put(17521,1129){\makebox(0,0)[lb]{\smash{{\SetFigFont{14}{16.8}{\familydefault}{\mddefault}{\updefault}{\color[rgb]{0,0,0}$u_{1}$}%
}}}}
\put(18644,1416){\makebox(0,0)[lb]{\smash{{\SetFigFont{14}{16.8}{\familydefault}{\mddefault}{\updefault}{\color[rgb]{0,0,0}$u_{2k}$}%
}}}}
\put(21369,2076){\makebox(0,0)[lb]{\smash{{\SetFigFont{14}{16.8}{\familydefault}{\mddefault}{\updefault}{\color[rgb]{0,0,0}$w_0$}%
}}}}
\put(19919,2236){\makebox(0,0)[lb]{\smash{{\SetFigFont{14}{16.8}{\familydefault}{\mddefault}{\updefault}{\color[rgb]{0,0,0}$v_{2k}$}%
}}}}
\put(18424,2121){\makebox(0,0)[lb]{\smash{{\SetFigFont{14}{16.8}{\familydefault}{\mddefault}{\updefault}{\color[rgb]{0,0,0}$u_{\alpha k +1}$}%
}}}}
\put(17516,1989){\makebox(0,0)[lb]{\smash{{\SetFigFont{14}{16.8}{\familydefault}{\mddefault}{\updefault}{\color[rgb]{0,0,0}$u_{\alpha k}$}%
}}}}
\put(16181,1949){\makebox(0,0)[lb]{\smash{{\SetFigFont{14}{16.8}{\familydefault}{\mddefault}{\updefault}{\color[rgb]{0,0,0}$v_{\beta k}$}%
}}}}
\put(18136,839){\makebox(0,0)[lb]{\smash{{\SetFigFont{14}{16.8}{\familydefault}{\mddefault}{\updefault}{\color[rgb]{0,0,0}$u_{-2k}{=}u_{2k+1}$}%
}}}}
\put(20431,974){\makebox(0,0)[lb]{\smash{{\SetFigFont{14}{16.8}{\familydefault}{\mddefault}{\updefault}{\color[rgb]{0,0,0}${=}v_{2k+1}$}%
}}}}
\end{picture}%

%% file: revised.bbl
\begin{thebibliography}{10}

\bibitem{Christian}
R.~Christian, R.~B. Richter, and G.~Salazar.
\newblock Zarankiewicz's conjecture is finite for each fixed {$m$}.
\newblock {\em J. Combin. Theory Ser. B}, 103(2):237--247, 2013.

\bibitem{elkies}
N.~Elkies.
\newblock Crossing numbers of complete graphs.
\newblock In {\em The mathematics of various entertaining subjects, Vol.~2},
  pages 218--249. Princeton Univ. Press, Princeton, NJ, 2017.

\bibitem{Guy10}
R.~K. Guy.
\newblock Latest results on crossing numbers.
\newblock In {\em Recent {T}rends in {G}raph {T}heory ({P}roc. {C}onf., {N}ew
  {Y}ork, 1970)}, pages 143--156. Lecture Notes in Mathematics, Vol. 186.
  Springer, Berlin, 1971.

\bibitem{guycrn}
R.~K. Guy.
\newblock Crossing numbers of graphs.
\newblock In {\em Graph theory and applications ({P}roc. {C}onf., {W}estern
  {M}ichigan {U}niv., {K}alamazoo, {M}ich., 1972; dedicated to the memory of
  {J}. {W}. {T}. {Y}oungs)}, pages 111--124, 1972.

\bibitem{HararyHill}
F.~Harary and A.~Hill.
\newblock On the number of crossings in a complete graph.
\newblock {\em Proc. Edinburgh Math. Soc. (2)}, 13:333--338, 1962/1963.

\bibitem{Kleitman}
D.~J. Kleitman.
\newblock The crossing number of {$K_{5,n}$}.
\newblock {\em J. Combinatorial Theory}, 9:315--323, 1970.

\bibitem{koman}
M.~Koman.
\newblock On the crossing numbers of graphs.
\newblock {\em Acta Univ. Carolinae--Math. et Phys.}, 10(nos. 1--2):9--46,
  1969.

\bibitem{DanBruce}
D.~McQuillan and R.~B. Richter.
\newblock On the crossing number of {$K_n$} without computer assistance.
\newblock {\em J. Graph Theory}, 82(4):387--432, 2016.

\bibitem{moon}
J.~W. Moon.
\newblock On the distribution of crossings in random complete graphs.
\newblock {\em J. Soc. Indust. Appl. Math.}, 13:506--510, 1965.

\bibitem{PanRichter}
S.~Pan and R.~B. Richter.
\newblock The crossing number of {$K_{11}$} is 100.
\newblock {\em J. Graph Theory}, 56(2):128--134, 2007.

\bibitem{MarcusBook}
M.~Schaefer.
\newblock {\em Crossing numbers of graphs}.
\newblock Discrete Mathematics and its Applications (Boca Raton). CRC Press,
  Boca Raton, FL, 2018.

\end{thebibliography}
